\documentclass[11pt]{amsart}
\usepackage{amsmath,amssymb,latexsym,cancel,rotating}

\numberwithin{equation}{section}

\newtheorem{theorem}{Theorem}[section]
\newtheorem{proposition}[theorem]{Proposition}
\newtheorem{conjecture}[theorem]{Conjecture}
\newtheorem{corollary}[theorem]{Corollary}
\newtheorem{lemma}[theorem]{Lemma}

\theoremstyle{definition}

\newtheorem{remark}[theorem]{Remark}

\newtheorem{definition}[theorem]{Definition}

\input xy
\xyoption{poly}
\xyoption{2cell}
\xyoption{all}

\def\ZZ{\mathbb{Z}}

\def\Acal{\mathcal{A}}

\def\Fcal{\mathcal{F}}

\def\Tcal{\mathcal{T}}

\def\Hom{\text{Hom}}
\def\Ext{\text{Ext}}

\renewcommand{\eqref}[1]{{\rm (\ref{#1})}}

\begin{document}

\title{A quantum analogue of generic bases for affine cluster algebras}

\author{Ming Ding and Fan Xu}
\address{Institute for advanced study\\
Tsinghua University\\
Beijing 100084, P.~R.~China} \email{m-ding04@mails.tsinghua.edu.cn
(M.Ding)}
\address{Department of Mathematical Sciences\\
Tsinghua University\\
Beijing 100084, P.~R.~China} \email{fanxu@mail.tsinghua.edu.cn
(F.Xu)}


\thanks{The research was
supported  by NSF of China (No. 11071133)}


\maketitle

\begin{abstract} We construct quantized versions of generic
bases in quantum cluster algebras of finite and affine types. Under
the specialization of  $q$ and coefficients to $1$, these bases are
generic bases of finite and affine cluster algebras.
\end{abstract}

\section{Introduction}
Cluster algebras were introduced by Fomin and Zelevinsky in order to
give a combinatorial framework for understanding total positivity in
algebraic groups and dual canonical bases in quantum groups
\cite{FZ}. Cluster algebras relate closely to algebraic
representation theory via two links. One is by cluster categories
(\cite{BMRRT}) and the Caldero-Chapoton map (\cite{caldchap}). The
other is by stable module categories over preprojective algebras
(\cite{GLS2006})  and evaluation forms (\cite{GLS2007}).

Let $Q=(Q_0, Q_1)$ be a finite quiver without oriented cycles where
$Q_0=\{1, 2, \cdots, n\}$ is the set of vertices and $Q_1$ is the
set of arrows. The (coefficient-free) cluster algebra
$\mathcal{A}(Q)$ is the subalgebra of $\mathbb{Q}(x_1, \cdots x_n)$
with a set of generators (called cluster variables). The cluster
category $\mathcal{C}_Q$ is the orbit category of $\mathcal{D}^b(Q)$
by the functor $\tau\circ [-1]$ (see \cite{BMRRT}). In
\cite{caldchap}, the authors introduced a map $X_{?}$ from the set
of objects in $\mathcal{C}_Q$ to $\mathbb{Z}[x^{\pm}_1, \cdots
x^{\pm}_n]$, called the Caldero-Chapoton map. There is a bijection
between the set of isoclasses of indecomposable rigid objects (i.e.,
without nontrivial self-extensions) in $\mathcal{C}_Q$ and the set
of cluster variables in $\mathcal{A}(Q)$ (see \cite [Theorem
4]{CK2}). Hence, the cluster algebra $\mathcal{A}(Q)$ can be viewed
as the subalgebra of $\mathbb{Z}[x^{\pm}_1, \cdots x^{\pm}_n]$
generated by  $X_M$ of indecomposable rigid objects $M$.

Quite differently, Geiss, Leclerc and Schr\"oer (\cite{GLS2006})
considered the preprojective algebra $\Lambda=\Lambda_Q$ of $Q$ and
the category $\mathrm{nil}(\Lambda)$ of finite dimensional nilpotent
$\Lambda$-modules. They (\cite{GLS2007}) attached to certain
preinjective representations $M$ of $Q$ a Frobenius subcategory
$\mathcal{C}_M$ of $\mathrm{nil}(\Lambda)$. If $M=I\oplus \tau(I)$
for the sum of indecomposable injective representations of $Q$, then
the stable category $\underline{\mathcal{C}}_M$ is triangle
equivalent to the cluster category $\mathcal{C}_Q.$ Let
$\mathcal{A}(\mathcal{C}_M)$ and
$\underline{\mathcal{A}}(\mathcal{C}_M)$ be the cluster algebra
associated to $\mathcal{C}_M$ and $\underline{\mathcal{C}}_M$,
respectively.

Let $\Lambda_{\textbf{d}}$ be the affine variety of nilpotent
$\Lambda$-modules with dimension vector $\textbf{d}=(d_i)_{i\in
Q_0}$. In \cite[Section 1.2]{GLS2007}, the authors defined the
evaluation form $\delta_M$ for $M\in \Lambda_{\textbf{d}}$. Define
$\langle M\rangle=\{N\in \Lambda_{\textbf{d}}\mid \delta_M=\delta_N
\}$. Let $Z$ be an irreducible component of $\Lambda_{\textbf{d}}.$
Then $M$ is a generic point in $Z$ if $\langle M\rangle\cap Z$
contains a dense open subset of $Z$. The involving dual
semicanonical basis is just the set of $\delta_{M}$ by taking a
generic point $M$ for each irreducible component $Z$. Geiss, Leclerc
and Schr\"oer (\cite{GLS2007}, \cite{GLS2010-1}) proved that
$\mathcal{A}(\mathcal{C}_M)$ is spanned by a subset of the dual
semicanonical basis and then specializes a basis of
$\underline{\mathcal{A}}(\mathcal{C}_M)$.

Analogously, Dupont (\cite{Dup}) introduced generic variables for
$\mathcal{A}(Q)$. Let $\mathbb{E}_{\textbf{d}}$ be the affine
variety of representations of $Q$ with dimension vector
$\textbf{d}.$ For $M\in \mathbb{E}_{\textbf{d}}$, define $\langle
M\rangle=\{N\in \mathbb{E}_{\textbf{d}}\mid X_M=X_N \}$. We say $M$
is generic if $\langle M\rangle$ contains a (dense) open subset of
$\mathbb{E}_{\textbf{d}}$ and then $X_{\textbf{d}}:=X_M$ is called
the generic variable for $\textbf{d}\in \mathbb{N}^n.$ The
definition can be naturally extended to $\textbf{d}\in
\mathbb{Z}^n.$ Set $\mathcal{G}(Q)=\{X_{\textbf{d}}\mid
\textbf{d}\in \mathbb{Z}^n\}.$ Dupont conjectured the set
$\mathcal{G}(Q)$ is the $\mathbb{Z}$-basis of $\mathcal{A}(Q)$ and
confirmed it for type $\widetilde{A}$. Recently, the conjecture has
been completely proved in \cite{GLS2010-2} (see also \cite{CK1} for
finite type, \cite{DXX} for affine type). Therefore, the set
$\mathcal{G}(Q)$ is called the generic basis of $\mathcal{A}(Q).$ In
particular, for an affine quiver $Q$, one obtain
$$
\mathcal{G}(Q)=\{X_{M}\mid M \text{ is rigid or } M\cong
nE_{\delta}\oplus R \mbox{ with } \mbox{ regular rigid }R \mbox{ and
regular simple } E_{\delta} \}.
$$

Compared with (dual) canonical basis, one would like to define the
quantized version of (dual) semicanonical basis for quantum groups.
The aim of the present note is to construct the quantized version of
generic bases for quantum cluster algebras of finite and affine
types. Quantum cluster algebras were introduced by A.~Berenstein and
A.~Zelevinsky \cite{berzel} as a noncommutative analogue of cluster
algebras \cite{ca1}\cite{ca2} to study canonical bases. A quantum
cluster algebra   is generated by a set of generators called the
\textit{quantum cluster variables} inside an ambient skew-field
$\mathcal{F}$. Under the specialization,  the quantum cluster
algebras are exactly cluster algebras which were introduced by
S.~Fomin and A.~Zelevinsky \cite{ca1}\cite{ca2}.

Let $q$ be a formal variable and $\mathcal{A}_q(Q)$ be the quantum
cluster algebra for $Q$ (see Section 2 for more details). Recently,
a quantum analog $X_{?}$ (we use the same notation as above without
causing confusion) of the Caldero-Chapoton map has been defined in
\cite{rupel} and refined in \cite{fanqin}. In \cite{fanqin}, the
author further showed that quantum cluster variables could be
expressed as images of indecomposable rigid objects under the
quantum Caldero-Chapoton map for acyclic equally valued quivers.

The note is organized as follows. In Section 2, we recall the
definition of quantum cluster algebras and the quantum
Caldero-Chapoton map. Section 3 is contributed to prove two
multiplication formulas for acyclic quantum cluster algebras. The
first formula (Theorem \ref{hall multi}) is a quantum version of the
formula $X_{M\oplus N}=X_MX_N$ in cluster algebras. It is also a
quantization of the degenerated form of Green's formula (See
\cite[Section 5]{Ding-Xiao-Xu}).  We will apply this formula to show
that the image of regular simple modules (over homogeneous tubes
with minimal imaginary root $\delta$ as dimension vector) under the
quantum Caldero-Chapoton map belong to quantum cluster algebras over
$\mathbb{Q}$ of affine type in Section 5. The second formula
(Theorem \ref{multi-formula} and \ref{exchange2}) generalizes the
quantum cluster multiplication formula in \cite{fanqin} to non-rigid
objects. The generalization is essential for this note. We apply
this formula to characterizing the regular modules over
non-homogeneous tubes with dimension vector $\delta$ (Lemma
\ref{k2}). In Section 4, we construct the quantum generic basis of a
quantum cluster algebra of finite type. By the specialization of $q$
and coefficients to $1$, the basis is no more than the good basis of
a finite-type cluster algebra in \cite{CK1}. The main results of
this note are shown in Section 5 and 6. We give a basis of an
affine-type quantum cluster algebras over $\mathbb{Q}$ in Section 5.
By the specialization of $q$ and coefficients to $1$, the basis is
just the generic basis of an affine-type cluster algebra in
\cite{DXX} and \cite{Dup}. In Section 6, we prove a formula for
quantum cluster algebras of type $\widetilde{A}$ and
$\widetilde{D}$. The formula characterizes the difference of regular
modules of dimension vector $\delta$ in a homogeneous tube and a
non-homogeneous tube. It is the quantum version of the difference
property in \cite{DXX} and \cite{Dup}. It helps us to refine the
basis  in Section 5 to integral bases.

\section{The quantum Caldero-Chapoton map}
\subsection{Quantum cluster algebras} The main reference for quantum cluster algebras is \cite{berzel}.
Here, we also recommend \cite[Section2]{fanqin} as a nice reference. Let $L$ be a lattice of rank $m$ and
$\Lambda:L\times L\to \ZZ$ a skew-symmetric bilinear form. Let $q$
be a formal variable and consider the ring of integer Laurent
polynomials $\ZZ[q^{\pm1/2}]$.  Define the \textit{based quantum
torus} associated to the pair $(L,\Lambda)$ to be the
$\ZZ[q^{\pm1/2}]$-algebra $\mathcal{T}$ with a distinguished
$\ZZ[q^{\pm1/2}]$-basis $\{X^e: e\in L\}$ and the  multiplication
given by
\[X^eX^f=q^{\Lambda(e,f)/2}X^{e+f}.\]
It is easy to see  that $\Tcal$ is associative and the basis
elements satisfy the following relations:
\[X^eX^f=q^{\Lambda(e,f)}X^fX^e,\ X^0=1,\ (X^e)^{-1}=X^{-e}.\]  It is known that $\Tcal$ is an Ore domain, i.e.,   is contained in its
skew-field of fractions $\Fcal$.  The quantum cluster algebra
 will be defined as a
$\ZZ[q^{\pm1/2}]$-subalgebra of $\Fcal$.

A \textit{toric frame} in $\Fcal$ is a map $M: \ZZ^m\to \Fcal
\setminus \{0\}$ of the form \[M({\bf c})=\varphi(X^{\eta({\bf
c})})\] where $\varphi$ is an automorphism of $\Fcal$ and $\eta:
\ZZ^m\to L$ is an  isomorphism of lattices.  By the definition, the
elements $M({\bf c})$ form a $\ZZ[q^{\pm1/2}]$-basis of the based
quantum torus $\Tcal_M:=\varphi(\Tcal)$ and satisfy the following
relations:
\[M({\bf c})M({\bf d})=q^{\Lambda_M({\bf c},{\bf d})/2}M({\bf c}+{\bf d}),\
M({\bf c})M({\bf d})=q^{\Lambda_M({\bf c},{\bf d})}M({\bf d})M({\bf
c}),\]
\[ M({\bf 0})=1,\ M({\bf c})^{-1}=M(-{\bf c}),\]
where $\Lambda_M$ is the skew-symmetric bilinear form on $\ZZ^m$
obtained from the lattice isomorphism $\eta$.  Let $\Lambda_M$ also
denote the skew-symmetric $m\times m$ matrix defined by
$\lambda_{ij}=\Lambda_M(e_i,e_j)$ where $\{e_1, \ldots, e_m\}$ is
the standard basis of $\ZZ^m$.  Given a toric frame $M$, let
$X_i=M(e_i)$.  Then we have
$$\Tcal_M=\ZZ[q^{\pm1/2}]\langle X_1^{\pm 1}, \ldots,
X_m^{\pm1}:X_iX_j=q^{\lambda_{ij}}X_jX_i\rangle.$$  An easy
computation shows that
\[M({\bf c})=q^{\frac{1}{2}\sum_{i<j}
c_ic_j\lambda_{ji}}X_1^{c_1}X_2^{c_2}\cdots X_m^{c_m}=:X^{{\bf c}} \
\ \ ({\bf c}\in\ZZ^m).\]

Let $\Lambda$ be an $m\times m$ skew-symmetric matrix and let
$\widetilde{B}$ be an $m\times n$ matrix for some positive integer
$n\leq m$. We call the pair $(\Lambda, \widetilde{B})$
\textit{compatible} if $\widetilde{B}^T\Lambda=(D|0)$ is an $n\times
m$ matrix with $D=diag(d_1,\cdots,d_n)$ where $d_i\in \mathbb{N}$
for $1\leq i\leq n$. The pair $(M,\widetilde{B})$ is called a
\textit{quantum seed} if the pair $(\Lambda_M, \widetilde{B})$ is
compatible.  Define the $m\times m$ matrix $E=(e_{ij})$ by
\[e_{ij}=\begin{cases}
\delta_{ij} & \text{if $j\ne k$;}\\
-1 & \text{if $i=j=k$;}\\
max(0,-b_{ik}) & \text{if $i\ne j = k$.}
\end{cases}
\]
For $n,k\in\ZZ$, $k\ge0$, denote ${n\brack
k}_q=\frac{(q^n-q^{-n})\cdots(q^{n-k+1}-q^{-n+k-1})}{(q^k-q^{-k})\cdots(q-q^{-1})}$.
Let ${\bf c}=(c_1,\ldots,c_m)\in\ZZ^m$ with $c_{k}\geq 0$.  Define
the toric frame $M': \ZZ^m\to \Fcal \setminus \{0\}$ as follows:
\begin{equation}\label{eq:cl_exp}M'({\bf c})=\sum^{c_k}_{p=0} {c_k \brack p}_{q^{d_k/2}} M(E{\bf c}+p{\bf b}^k),\ \ M'({\bf -c})=M'({\bf c})^{-1}.\end{equation}
where the vector ${\bf b}^k\in\ZZ^m$ is the $k-$th column of
$\widetilde{B}$.    Then the quantum seed $(M',\widetilde{B}')$ is
defined to be the mutation of $(M,\widetilde{B})$ in direction $k$.
In general, two quantum seeds $(M, \widetilde{B})$ and $(M',
\widetilde{B}')$ are mutation-equivalent if they can be obtained
from each other by a sequence of mutations, denoted by $(M,
\widetilde{B})\sim (M', \widetilde{B}')$. Let $\mathcal{C}=\{M'(e_i)
\mid (M, \widetilde{B})\sim (M', \widetilde{B}'), i=1, \cdots n\}$.
The elements of $\mathcal{C}$ are called \textit{quantum cluster
variables}. Let $\mathcal{P}=\{M(e_i): i=n+1, \cdots, m]\}$ and  the
elements of $\mathcal{P}$ are called \emph{coefficients}.  Given
$(M', \widetilde{B}')\sim (M, \widetilde{B})$ and ${\bf c}=(c_i)\in
\mathbb{Z}^m$, a element $M'(\bf c)$ is called  a \emph{quantum
cluster monomial} if $c_i\geq 0$ for $i=1, \cdots, n$ and $0$ for
$i=n+1, \cdots, m.$  We denote by $\mathbb{P}$ the multiplicative
group by $q^{\frac{1}{2}}$ and $\mathcal{P}$. Write $\mathbb{ZP}$ as
the ring of Laurent polynomials in the elements of $\mathcal{P}$
with coefficients in $\mathbb{Z}[q^{\pm 1/2}]$. Write $\mathbb{QP}$
as the ring of Laurent polynomials in the elements of $\mathcal{P}$
with coefficients in $\mathbb{Q}[q^{\pm 1/2}]$.  The \textit{quantum
cluster algebra} $\Acal_q(\Lambda_M,\widetilde{B})$ is the
$\mathbb{ZP}$-subalgebra of $\Fcal$ generated by $\mathcal{C}$. We
associate $(M,\tilde{B})$ a $\ZZ$-linear \emph{bar-involution} on
$\Tcal_M$ defined by
\[\overline{q^{r/2}M({\bf c})}=q^{-r/2}M({\bf c}), \ \  (r\in\ZZ,\
{\bf c}\in\ZZ^n).\]
 It is easy to show that
$\overline{XY}=\overline{Y}~\overline{X}$ for all $X,Y\in
\Acal_q(\Lambda_M, \widetilde{B})$ and that each element of
$\mathcal{C}\cup \mathcal{P}$ is \emph{bar-invariant}.

Now assume that there exists a finite field $k$ satisfying $|k|=q$.
In the same way, we can define based quantum torus
$\mathcal{T}_{|k|}$ and \emph{specialized quantum cluster algebras}
$\Acal_{|k|}(\Lambda_M,\widetilde{B})$ by substituting
$\mathbb{Z}[|k|^{\pm\frac{1}{2}}]$ for
$\mathbb{Z}[q^{\pm\frac{1}{2}}]$ in the above definition. By
\cite[Corollary 5.2]{berzel}, $\Acal_q(\Lambda_M, \widetilde{B})$
and $\Acal_{|k|}(\Lambda_M,\widetilde{B})$ are subalgebras of
$\mathcal{T}$ and $\mathcal{T}_{|k|}$, respectively. There is a
specialization map $ev: \mathcal{T}\rightarrow \mathcal{T}_{|k|}$ by
mapping $q^{\frac{1}{2}}$ to $|k|^{\frac{1}{2}}$, which induces a
bijection between quantum monomials of
$\Acal_{q}(\Lambda_M,\widetilde{B})$ and
$\Acal_{|k|}(\Lambda_M,\widetilde{B})$ (\cite[Section 2.2]{fanqin}).

\subsection{The quantum Caldero-Chapoton map}Let $k$ be a finite field with cardinality $|k|=q$ and
$m\geq n$ be two positive integers and $\widetilde{Q}$ an acyclic
quiver with vertex set $\{1,\ldots,m\}$ \cite{fanqin}. Denote the
subset $\{n+1,\dots,m\}$ by $C$. The elements in $C$ are called the
\emph{frozen vertices }, and $\widetilde{Q}$ is called an \emph{ice
quiver}. The full subquiver $Q$ on the vertices $1,\ldots,n$ is
called the \emph{principal part} of $\widetilde{Q}$.

Let $\widetilde{B}$ be the $m\times n$ matrix associated to the ice
quiver $\widetilde{Q}$, i.e., its entry in position $(i,j)$ is
\[
b_{ij}=|\{\mathrm{arrows}\, i\longrightarrow
j\}|-|\{\mathrm{arrows}\, j\longrightarrow i\}|
\]
for $1\leq i\leq m$, $1\leq j\leq n$. And let $\widetilde{I}$ be the
left $m\times n$ matrix of the identity matrix of size $m\times m$.
Further assume that there exists some antisymmetric $m\times m$
integer matrix $\Lambda$ such that
\begin{align}\label{eq:simply_laced_compatible}
\Lambda(-\widetilde{B})=\widetilde{I}:=\begin{bmatrix}I_n\\0
\end{bmatrix},
\end{align}
where $I_n$ is the identity matrix of size $n\times n$. Thus, the
matrix $\widetilde{B}$ is of full rank.

Let $\widetilde{R}$ and $\widetilde{R}^{tr}$ be the $m\times n$
matrix with its entry in position $(i,j)$ is
\[
\widetilde{r}_{ij}=\mathrm{dim}_{k}\mathrm{Ext}^{1}_{k\widetilde{Q}}(S_i,S_j)
\]
and
\[
\widetilde{r}^{*}_{ij}=\mathrm{dim}_{k}\mathrm{Ext}^{1}_{k\widetilde{Q}}(S_j,S_i)
\]
for $1\leq i\leq m$, $1\leq j\leq n$, respectively. Note that
\[
\mathrm{dim}_{k}\mathrm{Ext}^{1}_{k\widetilde{Q}}(S_i,S_j)=|\{\mathrm{arrows}\,
j\longrightarrow i\}|.
\]
 Denote the principal parts of
the matrices $\widetilde{B}$ and $\widetilde{R}$ by $B$ and $R$
respectively. Note that
$\widetilde{B}=\widetilde{R}^{tr}-\widetilde{R}$ and $B=R^{tr}-R$
where $R^{tr}$ represents the transposition of the matrix $R.$ In
general,  the matrix $B$ is not of full rank so that there exists no
matrix $\Lambda$ compatible with $B$. Hence, one need add  some
frozen vertices to $Q$ and then obtain an acyclic quiver
$\widetilde{Q}$ with a compatible pair $(\widetilde{B}, \Lambda).$

Let $\mathcal C_{\widetilde{Q}}$ be the cluster category of $k
\widetilde{Q}$, i.e., the orbit category of the derived category
$\mathcal{D}^b(\widetilde{Q})$ by the functor $F=\tau\circ[-1]$
where $\tau=\tau_{\widetilde{Q}}$ is the Auslander-Reiten
translation and $[1]$ is the translation functor. We note that the
indecomposable objects of the cluster category $\mathcal
C_{\widetilde{Q}}$ are either the indecomposable $k
\widetilde{Q}$-modules or $P_i[1]$ for indecomposable projective
modules $P_i$($1\leq i \leq m$).  Each object $M$ in $\mathcal
C_{\widetilde{Q}}$ can be uniquely decomposed in the following way:
$$M\cong M_0\oplus P_M[1]$$
where $M_0$ is a $k\widetilde{Q}$-module and $P_M$ is a projective
$k\widetilde{Q}$-module. Let $P_M=\bigoplus_{1\leq i \leq m}m_iP_i.$
We extend the definition of the dimension vector
$\mathrm{\underline{dim}}$ on modules in $\mathrm{mod}k
\widetilde{Q}$ to objects in $\mathcal C_{\widetilde{Q}}$ by setting
$$\mathrm{\underline{dim}}M=\mathrm{\underline{dim}}M_0-(m_i)_{1\leq i \leq m}.$$
The Euler form on $k\widetilde{Q}$-modules $M$ and $N$ is given by
$$\langle M,N\rangle=\mathrm{dim}_{k}\mathrm{Hom}_{k\widetilde{Q}}(M,N)-\mathrm{dim}_{k}\mathrm{Ext}^{1}_{k\widetilde{Q}}(M,N).$$
Note that the Euler form only depends on the dimension vectors of
$M$ and $N$. As in \cite{Hubery}, we define
$$
[M, N]=\mathrm{dim}_{k}\mathrm{Hom}_{k\widetilde{Q}}(M,N)\mbox{ and
}[M, N]^1=\mathrm{dim}_{k}\mathrm{Ext}^{1}_{k\widetilde{Q}}(M,N).
$$

The quantum Caldero-Chapoton map of an acyclic quiver
$\widetilde{Q}$ has been studied in \cite{rupel} and \cite{fanqin}.
Here, we reformulate their definitions to the following map
$$X^{\widetilde{Q}}_?: \mathrm{obj}\mathcal C_{\widetilde{Q}}\longrightarrow \Tcal$$
defined by the following rule: If $M$ is a $k Q$-module and $P$ is a
projective $k \widetilde{Q}$-module, then
                    $$
                       X^{\widetilde{Q}}_{M\oplus P[1]}=\sum_{\underline{e}} |\mathrm{Gr}_{\underline{e}} M|q^{-\frac{1}{2}
\langle
\underline{e},\underline{m}-\underline{e}\rangle}X^{\widetilde{B}\underline{e}-(\widetilde{I}-\widetilde{R})\underline{m}+\underline{\mathrm{dim}}
P/\mathrm{rad}P},
                    $$
where $\underline{\mathrm{dim}} M= \underline{m}$ and
$\mathrm{Gr}_{\underline{e}}M$ denotes the set of all submodules $V$
of $M$ with $\underline{\mathrm{dim}} V= \underline{e}$. Usually, we
omit the upper index $\widetilde{Q}$ in the notation
$X^{\widetilde{Q}}_?$ (except Section 4 and Section 5). It should
not cause any confusion. We note that
$$
X_{P[1]}=X_{\tau P}=X^{\underline{\mathrm{dim}} P/rad
P}=X^{\underline{\mathrm{dim}}\mathrm{soc}I}=X_{I[-1]}=X_{\tau^{-1}I}.
$$
for any projective $k\widetilde{Q}$-module $P$ and injective
$k\widetilde{Q}$-module $I$ with $\mathrm{soc}I=P/\mathrm{rad}P.$
Hereinafter, we denote by the corresponding underlined small letter
$\underline{x}$ the dimension vector of a $kQ$-module $X$ and view
$\underline{x}$ as a column vector in $\mathbb{Z}^n.$

\section{Multiplication theorems for acyclic quantum cluster algebras}
Throughout this section, assume that $\widetilde{Q}$ is an acyclic
quiver and $Q$ is its full subquiver. In this section, we will prove
a multiplication theorem for any acyclic quantum cluster algebra.
First, we improve Lemma 5.2.1 and Corollary 5.2.2 in \cite{fanqin},
i.e., here we handle the dimension vector of any $kQ$-module while
in \cite{fanqin} the author only deals with dimension vectors of
rigid modules.
\begin{lemma}\label{1}
For any dimension vector $\underline{m}, \underline{e},
\underline{f}\in \mathbb{Z}^{n}_{\geq 0},$ we have
$$(1)\ \Lambda((\widetilde{I}-\widetilde{R})\underline{m}, \widetilde{B}\underline{e})=-\langle \underline{e}, \underline{m}\rangle;$$
$$(2)\ \Lambda(\widetilde{B}\underline{e}, \widetilde{B}\underline{f})=\langle \underline{e}, \underline{f}\rangle-\langle \underline{f}, \underline{e}\rangle.$$
\end{lemma}
\begin{proof}
By definition, we have \begin{eqnarray}
   && \Lambda((\widetilde{I}-\widetilde{R})\underline{m}, \widetilde{B}\underline{e})  \nonumber\\
   &=& \underline{m}^{tr}(\widetilde{I}-\widetilde{R})^{tr}\Lambda \widetilde{B}\underline{e}=-\underline{m}^{tr}(\widetilde{I}-\widetilde{R})^{tr}\begin{bmatrix}I_n\\0 \end{bmatrix}\underline{e}\nonumber\\
   &=& -\underline{m}^{tr}(I_{n}-R)^{tr}\underline{e}=-\underline{e}^{tr}(I_{n}-R)\underline{m}\nonumber\\
  &=& -\langle \underline{e}, \underline{m}\rangle.\nonumber
\end{eqnarray}
As for (2), the left side of the desired equation is equal to
$$\underline{e}^{tr}\widetilde{B}^{tr}\Lambda
\widetilde{B}\underline{f}=-\underline{e}^{tr}\widetilde{B}^{tr}\begin{bmatrix}I_n\\0
\end{bmatrix}\underline{f}=-\underline{e}^{tr}B^{tr}\underline{f}.$$
The right side is
\begin{eqnarray}
   && \langle \underline{e}, \underline{f}\rangle-\langle \underline{f}, \underline{e}\rangle  \nonumber\\
   &=& \underline{e}^{tr}(I_{n}-R)\underline{f}-\underline{f}^{tr}(I_{n}-R)\underline{e}\nonumber\\
   &=& \underline{e}^{tr}(I_{n}-R)\underline{f}-\underline{e}^{tr}(I_{n}-R)^{tr}\underline{f}\nonumber\\
  &=& \underline{e}^{tr}(R^{tr}-R)\underline{f}=-\underline{e}^{tr}(R-R^{tr})\underline{f}=-\underline{e}^{tr}B^{tr}\underline{f}.\nonumber
\end{eqnarray}
Thus we prove the lemma.
\end{proof}
\begin{corollary}\label{2}
For any dimension vector $\underline{m}, \underline{l},
\underline{e}, \underline{f}\in \mathbb{Z}^{n}_{\geq 0},$ we have
\begin{eqnarray}
   && \Lambda(\widetilde{B}\underline{e}-(\widetilde{I}-\widetilde{R})\underline{m},\widetilde{B}\underline{f}-(\widetilde{I}-\widetilde{R})\underline{l})  \nonumber\\
  &=&\Lambda((\widetilde{I}-\widetilde{R})\underline{m},(\widetilde{I}-\widetilde{R})\underline{l})+\langle \underline{e}, \underline{f}\rangle-\langle \underline{f}, \underline{e}\rangle-\langle \underline{e}, \underline{l}\rangle+\langle \underline{f}, \underline{m}\rangle.\nonumber
\end{eqnarray}
\end{corollary}

For any $kQ-$modules $M,N,E$, denote by $\varepsilon_{MN}^{E}$  the
cardinality of the set $\mathrm{Ext}^{1}_{kQ}(M,N)_{E}$ which is the
subset of $ \mathrm{Ext}^{1}_{kQ}(M,N)$ consisting of those
equivalence classes of short exact sequences with middle term
isomorphic to $M$ (\cite[Section 4]{Hubery}). For $kQ$-modules $M$,
$A$ and $B$, we denote by $F^M_{AB}$ the number of submodules $U$ of
$M$ such that $U$ is isomorphic to $B$ and $M/U$ is isomorphic to
$A$. Then by definition, we have
$$|\mathrm{Gr}_{\underline{e}}(M)|=\sum_{A, B;
\underline{\mathrm{dim}}B=\underline{e}}F_{AB}^M.
$$ Different from the case in cluster categories, for $kQ$-modules,
it does not generally hold  that $X_{N}X_{M}=X_{N\oplus M}.$ We have
the following explicit characterization, which is a generalization
of \cite[Proposition 5.3.2]{fanqin}.
\begin{theorem}\label{hall multi}
Let $M$ and $N$ be $kQ$-modules. Then
$$q^{[M,N]^{1}}X_{N}X_{M}=q^{-\frac{1}{2}\Lambda((\widetilde{I}-\widetilde{R})\underline{m},
(\widetilde{I}-\widetilde{R})\underline{n})}
\sum_{E}\varepsilon_{MN}^{E}X_E.$$
\end{theorem}
\begin{proof}
We apply Green's formula in \cite{Green}
$$\sum_{E}\varepsilon_{MN}^{E}F^{E}_{XY}=\sum_{A,B,C,D}q^{[M,N]-[A,C]-[B,D]-\langle
A,D\rangle}F^{M}_{AB}F^{N}_{CD}\varepsilon_{AC}^{X}\varepsilon_{BD}^{Y}.$$
Then
\begin{eqnarray}
   && \sum_{E}\varepsilon_{MN}^{E}X_E  \nonumber\\
   &=& \sum_{E,X,Y}\varepsilon_{MN}^{E}q^{-\frac{1}{2}\langle Y,X\rangle}F^{E}_{XY}X^{\widetilde{B}\underline{y}-(\widetilde{I}-\widetilde{R})\underline{e}}\nonumber\\
  &=&\sum_{A,B,C,D,X,Y}q^{[M,N]-[A,C]-[B,D]-\langle
A,D\rangle-\frac{1}{2}\langle
B+D,A+C\rangle}F^{M}_{AB}F^{N}_{CD}\varepsilon_{AC}^{X}\varepsilon_{BD}^{Y}X^{\widetilde{B}\underline{y}-(\widetilde{I}-\widetilde{R})\underline{e}}.\nonumber
\end{eqnarray}
Since
\begin{eqnarray}
   && X^{\widetilde{B}\underline{y}-(\widetilde{I}-\widetilde{R})\underline{e}} \nonumber\\
   &=& X^{\widetilde{B}(\underline{b}+\underline{d})-(\widetilde{I}-\widetilde{R})(\underline{m}+\underline{n})}\nonumber\\
   &=& q^{-\frac{1}{2}\Lambda(\widetilde{B}\underline{d}-(\widetilde{I}-\widetilde{R})\underline{n},
\widetilde{B}\underline{b}-(\widetilde{I}-\widetilde{R})\underline{m})}X^{\widetilde{B}\underline{d}-(\widetilde{I}-\widetilde{R})\underline{n}}
X^{\widetilde{B}\underline{b}-(\widetilde{I}-\widetilde{R})\underline{m}}\nonumber\\
   &=& q^{-\frac{1}{2}\Lambda((\widetilde{I}-\widetilde{R})\underline{n},
(\widetilde{I}-\widetilde{R})\underline{m})-\frac{1}{2}[\langle
D,B\rangle-\langle B,D\rangle-\langle D,M\rangle+\langle
B,N\rangle]}X^{\widetilde{B}\underline{d}-(\widetilde{I}-\widetilde{R})\underline{n}}
X^{\widetilde{B}\underline{b}-(\widetilde{I}-\widetilde{R})\underline{m}}\nonumber\\
  &=&q^{-\frac{1}{2}\Lambda((\widetilde{I}-\widetilde{R})\underline{n},
(\widetilde{I}-\widetilde{R})\underline{m})}q^{\frac{1}{2}\langle
D,A\rangle-\frac{1}{2}\langle
B,C\rangle}X^{\widetilde{B}\underline{d}-(\widetilde{I}-\widetilde{R})\underline{n}}
X^{\widetilde{B}\underline{b}-(\widetilde{I}-\widetilde{R})\underline{m}}.\nonumber
\end{eqnarray}
Thus
\begin{eqnarray}
   && \sum_{E}\varepsilon_{MN}^{E}X_E  \nonumber\\
   &=& q^{\frac{1}{2}\Lambda((\widetilde{I}-\widetilde{R})\underline{m},
(\widetilde{I}-\widetilde{R})\underline{n})}\sum_{A,B,C,D}q^{[M,N]-[A,C]-[B,D]-\langle
A,D\rangle-\frac{1}{2}\langle
B+D,A+C\rangle+[A,C]^{1}+[B,D]^{1}}\cdot\nonumber\\
  &&q^{\frac{1}{2}\langle
D,A\rangle-\frac{1}{2}\langle
B,C\rangle}F^{M}_{AB}F^{N}_{CD}X^{\widetilde{B}\underline{d}-(\widetilde{I}-\widetilde{R})\underline{n}}
X^{\widetilde{B}\underline{b}-(\widetilde{I}-\widetilde{R})\underline{m}}.\nonumber
\end{eqnarray}
Here we use the following fact
$$\sum_{X}\varepsilon_{AC}^{X}=q^{[A,C]^{1}},\sum_{Y}\varepsilon_{BD}^{Y}=q^{[B,D]^{1}}$$
Note that
$$[M,N]-[A,C]-[B,D]-\langle A,D\rangle+[A,C]^{1}+[B,D]^{1}=[M,N]^{1}+\langle B, C\rangle.$$
Hence
\begin{eqnarray}
   && \sum_{E}\varepsilon_{MN}^{E}X_E  \nonumber\\
   &=& q^{\frac{1}{2}\Lambda((\widetilde{I}-\widetilde{R})\underline{m},
(\widetilde{I}-\widetilde{R})\underline{n})}q^{[M,N]^{1}}\sum_{A,B,C,D}q^{\langle
B,C\rangle-\frac{1}{2}\langle B,C\rangle-\frac{1}{2}\langle
D,A\rangle+\frac{1}{2}\langle D,A\rangle-\frac{1}{2}\langle
B,C\rangle}\cdot\nonumber\\
  && F^{N}_{CD}q^{-\frac{1}{2}\langle D,C\rangle}X^{\widetilde{B}\underline{d}-(\widetilde{I}-\widetilde{R})\underline{n}}
F^{M}_{AB}q^{-\frac{1}{2}\langle
B,A\rangle}X^{\widetilde{B}\underline{b}-(\widetilde{I}-\widetilde{R})\underline{m}}
\nonumber\\
 &=&q^{\frac{1}{2}\Lambda((\widetilde{I}-\widetilde{R})\underline{m},
(\widetilde{I}-\widetilde{R})\underline{n})}q^{[M,N]^{1}}X_{N}X_{M}.\nonumber
\end{eqnarray}
This completes the proof.
\end{proof}

\begin{remark}
Theorem \ref{hall multi} is similar to the multiplication formula in
dual Hall algebras. It is reasonable to conjecture that it provides
some PBW-type basis (\cite{GP}) in the corresponding quantum cluster
algebra.
\end{remark}

Let $M,N$ be $kQ-$modules and  assume that
$$\mathrm{dim}_{k}\mathrm{Ext}^{1}_{k\widetilde{Q}}(M,N)=\mathrm{dim}_{k}\mathrm{Hom}_{k\widetilde{Q}}(N,\tau M)=1.$$
Then there are two ``canonical'' exact
sequences
$$\varepsilon:\quad 0\longrightarrow N\longrightarrow E\longrightarrow M\longrightarrow 0$$
$$\varepsilon': \quad 0\longrightarrow D_{0}\longrightarrow N\longrightarrow \tau M\longrightarrow \tau A\oplus I\longrightarrow 0$$
which induces the $k$-bases of
$\mathrm{Ext}^{1}_{k\widetilde{Q}}(M,N)$ and
$\mathrm{Hom}_{k\widetilde{Q}}(N,\tau M)$, respectively. We fix
them. Set $M=M'\oplus P_0, A_0=A\oplus P_0$ where $P_0$ is a
projective $k\widetilde{Q}$-module, $A$ and $M'$ have no projective
summands. The exact sequences also provide the two non-split
triangles in $\mathcal{C}_{\widetilde{Q}}$:
$$N\longrightarrow E\longrightarrow M\longrightarrow N[1]=\tau N$$
and
$$M\longrightarrow D_{0}\oplus  A_0\oplus I[-1]\longrightarrow N\longrightarrow \tau M.$$

Now we state the first part of our multiplication theorem for
acyclic quantum cluster algebras, which can be viewed as a quantum
analogue of the one-dimensional Caldero-Keller multiplication
theorem in \cite{CK2}. The main idea in the proof comes from
\cite{Hubery}.
\begin{theorem}\label{multi-formula}
With the above notation,  assume that
$\mathrm{Hom}_{k\widetilde{Q}}(D_0,\tau A_0\oplus
I)=\mathrm{Hom}_{k\widetilde{Q}}(A_0,I)=0.$ Then the following
formula holds
$$ X_{N}X_M=q^{\frac{1}{2}\Lambda((\widetilde{I}-\widetilde{R})\underline{n},
(\widetilde{I}-\widetilde{R})\underline{m})}X_E+q^{\frac{1}{2}\Lambda((\widetilde{I}-\widetilde{R})\underline{n},
(\widetilde{I}-\widetilde{R})\underline{m})+\frac{1}{2}\langle
M,N\rangle-\frac{1}{2}\langle A_0, D_0\rangle}X_{D_0\oplus A_0\oplus
I[-1]}.$$
\end{theorem}
Here, we note that
$$\frac{q^{{[M,N]^1}}-1}{q-1}X_{N}X_M=X_{N}X_M.$$
\begin{proof}
By definition, we have
\begin{eqnarray}
   && X_{N}X_M  \nonumber\\
   &=& \sum_{C,D}q^{-\frac{1}{2}\langle D,C\rangle}F^{N}_{CD}X^{\widetilde{B}\underline{d}-(\widetilde{I}-\widetilde{R})\underline{n}}
   \sum_{A,B}q^{-\frac{1}{2}\langle B,A\rangle}F^{M}_{AB}X^{\widetilde{B}\underline{b}-(\widetilde{I}-\widetilde{R})\underline{m}}\nonumber\\
  &=& \sum_{A,B,C,D}F^{M}_{AB}F^{N}_{CD}q^{-\frac{1}{2}\langle D,C\rangle-\frac{1}{2}\langle B,A\rangle+\frac{1}{2}\Lambda(\widetilde{B}\underline{d}-
  (\widetilde{I}-\widetilde{R})\underline{n},
\widetilde{B}\underline{b}-(\widetilde{I}-\widetilde{R})\underline{m})}X^{\widetilde{B}(\underline{b}+\underline{d})-(\widetilde{I}-\widetilde{R})(\underline{m}+\underline{n})}
\nonumber\\
 &=&\sum_{A,B,C,D}F^{M}_{AB}F^{N}_{CD}q^{-\frac{1}{2}\langle B+D,A+C\rangle}q^{\langle B,C\rangle}q^{\frac{1}{2}\Lambda((\widetilde{I}-\widetilde{R})\underline{n},
(\widetilde{I}-\widetilde{R})\underline{m})}X^{\widetilde{B}(\underline{b}+\underline{d})-(\widetilde{I}-\widetilde{R})(\underline{m}+\underline{n})}.\nonumber
\end{eqnarray}

We set
$$s_1:=\sum_{E\ncong M\oplus N}\frac{\varepsilon_{MN}^{E}}{q-1}X_E=\sum_{X,Y,E\ncong M\oplus N}
\frac{\varepsilon_{MN}^{E}}{q-1}F^{E}_{XY}q^{-\frac{1}{2}\langle Y,X\rangle}
   X^{\widetilde{B}\underline{y}-(\widetilde{I}-\widetilde{R})\underline{e}}$$
As in the proof of Theorem \ref{hall multi}, we have
\begin{eqnarray}
   && \sum_{X,Y,E}\varepsilon_{MN}^{E}X_E  \nonumber\\
   &=& \sum_{A,B,C,D,X,Y}q^{[M,N]-[A,C]-[B,D]-\langle
A,D\rangle-\frac{1}{2}\langle
B+D,A+C\rangle}F^{M}_{AB}F^{N}_{CD}\varepsilon_{AC}^{X}\varepsilon_{BD}^{Y}X^{\widetilde{B}\underline{y}-(\widetilde{I}-\widetilde{R})\underline{e}}\nonumber\\
  &=& \sum_{A,B,C,D}q^{[M,N]^{1}+\langle
B,C\rangle-\frac{1}{2}\langle
B+D,A+C\rangle}F^{M}_{AB}F^{N}_{CD}X^{\widetilde{B}\underline{y}-(\widetilde{I}-\widetilde{R})\underline{e}}.\nonumber
\end{eqnarray}
On the other hand
$$X_{M\oplus N}=\sum_{A,B,C,D}q^{[B,C]-\frac{1}{2}\langle
B+D,A+C\rangle}F^{M}_{AB}F^{N}_{CD}X^{\widetilde{B}(\underline{b}+\underline{d})-(\widetilde{I}-\widetilde{R})(\underline{m}+\underline{n})}.$$
Thus
$$s_1=\sum_{A,B,C,D}\frac{q^{[M,N]^{1}}-q^{[B,C]^{1}}}{q-1}q^{\langle B,C\rangle-\frac{1}{2}\langle
B+D,A+C\rangle}F^{M}_{AB}F^{N}_{CD}X^{\widetilde{B}(\underline{b}+\underline{d})-(\widetilde{I}-\widetilde{R})(\underline{m}+\underline{n})}.$$

Thirdly we compute the term
$$s_2:=\sum_{A,D_0,I,D_0\ncong N}\frac{|\mathrm{Hom}_{k\widetilde{Q}}(N,\tau M)_{D_{0}AI}|}{q-1}X_{A_{0}\oplus D_{0}\oplus I[-1]}.$$
Here, we use the following notation as in \cite{Hubery}
$$\mathrm{Hom}_{k\widetilde{Q}}(N,\tau M)_{D_0AI}:=\{f\neq 0: N\longrightarrow \tau
M|\mathrm{ker}f\cong D_0, \mathrm{coker}f\cong \tau A\oplus I\}.$$
Note that $\mathrm{dim}_{k}\mathrm{Hom}_{k\widetilde{Q}}(N,\tau
M)=1,$ we have the following exact sequences
$$0\longrightarrow B_0\longrightarrow M\longrightarrow A_0\longrightarrow 0$$
$$0\longrightarrow C\longrightarrow \tau B_0\longrightarrow I\longrightarrow 0$$
where $C=\mathrm{im}f, \mathrm{ker}f=D_0.$

\begin{eqnarray}
   s_2&=& \frac{|\mathrm{Hom}_{k\widetilde{Q}}(N,\tau M)|-1}{q-1}X_{A_0\oplus D_0\oplus I[-1]}  \nonumber\\
   &=& \sum_{X,Y,K,L}\frac{|\mathrm{Hom}_{k\widetilde{Q}}(N,\tau M)|-1}{q-1}F^{D_0}_{XY}F^{A_0}_{KL}q^{[L,X]-\frac{1}{2}\langle
Y+L,K+X\rangle}X^{\widetilde{B}(\underline{y}+\underline{l}+\underline{b_0})-(\widetilde{I}-\widetilde{R})(\underline{m}+\underline{n})}\nonumber\\
  &=&\sum_{A,B,C,D} \frac{q^{[C,\tau B]}-1}{q-1}F^{M}_{AB}F^{N}_{CD}q^{[L,X]-\frac{1}{2}\langle
Y+L,K+X\rangle}X^{\widetilde{B}(\underline{y}+\underline{l}+\underline{b_0})-(\widetilde{I}-\widetilde{R})(\underline{m}+\underline{n})}.\nonumber
\end{eqnarray}
Here, $Y=D,K=A,B=B_0+L$ in the above expression and the equality can
be illustrated by the following diagram:

$$
\xymatrix{&Y\ar@{=}[r]\ar[d]&Y\ar[d]&\tau A\ar@{=}[r]&\tau A&\\
 0\ar[r]&D_0\ar[r]\ar[d]&N\ar[r]\ar[d]&\tau
M\ar[r]\ar[u]&\tau A_0\oplus I\ar[r]\ar[u]&0\\
0\ar[r]&X\ar[r]&C\ar[r]&\tau B\ar[r]\ar[u]&\tau L\oplus
I\ar[r]\ar[u]&0}
$$
We must to check the relation between
$$-\frac{1}{2}\langle Y+L,K+X\rangle+[L,X]$$ and
$$-\frac{1}{2}\langle B+D,A+C\rangle+\langle B,C\rangle.$$
In this case, note that $ D=Y, L=A_0-A,K=A,[L,X]^{1}=[X,\tau L]=0.$
We have
\begin{eqnarray}
  -\frac{1}{2}\langle Y+L,K+X\rangle+[L,X] &=& -\frac{1}{2}\langle Y+L,K+X\rangle+\langle L,X\rangle \nonumber\\
   &=& -\frac{1}{2}\langle D+A_0-A,A+D_0-D\rangle+\langle A-A_0,D_0-D\rangle\nonumber\\
   &=& -\frac{1}{2}\langle D,A\rangle-\frac{1}{2}\langle D,D_0\rangle+\frac{1}{2}\langle D,D\rangle-\frac{1}{2}\langle A_0,A\rangle
   +\frac{1}{2}\langle A_0,D_0\rangle\nonumber\\
   && -\frac{1}{2}\langle A_0,D\rangle+\frac{1}{2}\langle A,A\rangle
   -\frac{1}{2}\langle A,D_0\rangle+\frac{1}{2}\langle A,D\rangle. \nonumber\
\end{eqnarray}
And
\begin{eqnarray}
   && -\frac{1}{2}\langle B+D,A+C\rangle+\langle B,C\rangle \nonumber\\
   &=& -\frac{1}{2}\langle M-A+D,A+N-D\rangle+\langle M-A,N-D\rangle\nonumber\\
   &=& -\frac{1}{2}\langle M,A\rangle-\frac{1}{2}\langle M,D\rangle+\frac{1}{2}\langle A,A\rangle-\frac{1}{2}\langle A,N\rangle+\frac{1}{2}\langle A,D\rangle\nonumber\\
   && -\frac{1}{2}\langle D,A\rangle-\frac{1}{2}\langle D,N\rangle+\frac{1}{2}\langle D,D\rangle+\frac{1}{2}\langle M,N\rangle. \nonumber\
\end{eqnarray}
Hence it is equivalent to compare
$$-\frac{1}{2}\langle D,D_0\rangle-\frac{1}{2}\langle A_0,A\rangle+\frac{1}{2}\langle A_0,D_0\rangle-\frac{1}{2}\langle A_0,D\rangle-\frac{1}{2}\langle A,D_0\rangle$$
and
$$-\frac{1}{2}\langle D,N\rangle-\frac{1}{2}\langle M,A\rangle+\frac{1}{2}\langle M,N\rangle-\frac{1}{2}\langle M,D\rangle-\frac{1}{2}\langle A,N\rangle.$$
We claim that
$$\langle D,N\rangle+\langle M,D\rangle=\langle D,D_0\rangle+\langle A_0,D\rangle$$
and
$$\langle A_0,A\rangle+\langle A,D_0\rangle=\langle M,A\rangle+\langle A,N\rangle.$$
Indeed, we have
\begin{eqnarray}
  \langle D,N-D_0\rangle &=& \langle D,\tau M-\tau A_0-I\rangle \nonumber\\
   &=& \langle D,\tau M-\tau A_0\rangle=\langle A_0-M,D\rangle. \nonumber\
\end{eqnarray}
In the same way, we also have
$$
\langle A, N-D_0\rangle=\langle A_0-M, A\rangle.
$$

 Thus
$$s_2=q^{\frac{1}{2}\langle
A_0,D_0\rangle-\frac{1}{2}\langle
M,N\rangle}\sum_{A,B,C,D}\frac{q^{[B,C]^{1}}-1}{q-1}q^{\langle
B,C\rangle-\frac{1}{2}\langle
B+D,A+C\rangle}F^{M}_{AB}F^{N}_{CD}X^{\widetilde{B}(\underline{b}+\underline{d})-(\widetilde{I}-\widetilde{R})(\underline{m}+\underline{n})}.$$

Therefore we have the following multiplication formula
$$ X_{N}X_M=q^{\frac{1}{2}\Lambda((\widetilde{I}-\widetilde{R})\underline{n},
(\widetilde{I}-\widetilde{R})\underline{m})}X_E+q^{\frac{1}{2}\Lambda((\widetilde{I}-\widetilde{R})\underline{n},
(\widetilde{I}-\widetilde{R})\underline{m})+\frac{1}{2}\langle
M,N\rangle-\frac{1}{2}\langle A_0,D_0\rangle}X_{D_0\oplus A_0\oplus
I[-1]}.$$
\end{proof}

There are three canonical special cases satisfying the assumption
$\mathrm{Hom}_{k\widetilde{Q}}(D_0,\tau A\oplus
I)=\mathrm{Hom}_{k\widetilde{Q}}(A_0,I)=0$ in Theorem
\ref{multi-formula}.

\noindent\textbf{Special case I}. Assume that $A_0=0=I.$ Then
$L=K=0=A.$ If $B\neq M,$ i.e, $B\subsetneqq M,$ then there exists
$f_1: N\longrightarrow \tau M$ induced by the above diagram which is
not surjective. It is a contradiction to the assumption
$\mathrm{dim}_{k}\mathrm{Hom}_{k\widetilde{Q}}(N,\tau M)=1.$ In this
case, the multiplication formula is
$$ X_{N}X_M=q^{\frac{1}{2}\Lambda((\widetilde{I}-\widetilde{R})\underline{n},
(\widetilde{I}-\widetilde{R})\underline{m})}X_E+q^{\frac{1}{2}\Lambda((\widetilde{I}-\widetilde{R})\underline{n},
(\widetilde{I}-\widetilde{R})\underline{m})+\frac{1}{2}\langle
M,N\rangle}X_{D_0}.$$

\noindent\textbf{Special case II}. Assume that $D_0=0$ and
$\mathrm{Hom}_{k\widetilde{Q}}(A_0,I)=0.$ Then $Y=X=0, C=N.$ In this
case, the multiplication formula is
$$ X_{N}X_M=q^{\frac{1}{2}\Lambda((\widetilde{I}-\widetilde{R})\underline{n},
(\widetilde{I}-\widetilde{R})\underline{m})}X_E+q^{\frac{1}{2}\Lambda((\widetilde{I}-\widetilde{R})\underline{n},
(\widetilde{I}-\widetilde{R})\underline{m})+\frac{1}{2}\langle
M,N\rangle}X_{A_0\oplus I[-1]}.$$

\noindent\textbf{Special case III}. Assume that $M,N$ are
indecomposable rigid $kQ$-mod and
$$\mathrm{dim}_k\mathrm{Ext}^1_{\mathcal{C}_{\widetilde{Q}}}(M, N)=1.$$
Since $D_0\oplus A_0\oplus I[-1]$ is rigid, then the assumption
$\mathrm{Hom}_{k\widetilde{Q}}(D_0,\tau A\oplus
I)=\mathrm{Hom}_{k\widetilde{Q}}(A_0,I)=0$ in Theorem
\ref{multi-formula} holds.
\begin{lemma}\label{special}
With the assumption in Special case III, we have $\frac{1}{2}\langle
A_0,D_0\rangle-\frac{1}{2}\langle M,N\rangle=\frac{1}{2}.$
\end{lemma}
\begin{proof}
Note that we have $$\frac{1}{2}\langle
A_0,D_0\rangle-\frac{1}{2}\langle M,N\rangle=\frac{1}{2}\langle
A_0,N-N/D_0\rangle-\frac{1}{2}\langle M,N\rangle.$$ We need to
confirm the two equations
\begin{enumerate}
  \item $\langle
M,N\rangle=\langle A_0,N\rangle-1$ and
  \item $\langle
A_0,N/D_0\rangle=0.$
\end{enumerate}
Note that $A_0\oplus N$ is rigid, thus $[A_0,N]^{1}=0.$ We have the
following exact sequences
$$0\longrightarrow N/D_0\longrightarrow \tau M\longrightarrow \tau A_0 \oplus I\longrightarrow 0$$
$$0\longrightarrow D_0\longrightarrow N\longrightarrow N/D_0\longrightarrow 0$$
Applying the functor $\mathrm{Hom}_{k\widetilde{Q}}(N,-)$, we have
the following exact sequences
$$[N,N/D_0]^{1}\longrightarrow [N,\tau M]^{1}\longrightarrow [N,\tau A_0\oplus I]^{1}\longrightarrow 0$$
$$[N,N]^{1}\longrightarrow [N,N/D_0]^{1}\longrightarrow 0.$$
Hence
$$\langle
M,N\rangle=[M,N]-1=[A_0,N]-1=\langle A_0,N\rangle-1.$$ As for the
second equation, apply the functor
$\mathrm{Hom}_{k\widetilde{Q}}(A_0,-)$ to the exact sequence
$$0\longrightarrow D_0\longrightarrow N\longrightarrow N/D_0\longrightarrow 0$$
We have the following exact sequence
$$[A_0,N]^{1}\longrightarrow [A_0,N/D_0]^{1}\longrightarrow 0$$
Thus $[A_0,N/D_0]^{1}=0.$ Applying the functor
$\mathrm{Hom}_{k\widetilde{Q}}(\tau M,-)$ to the exact sequence
$$0\longrightarrow N/D_0\longrightarrow \tau M\longrightarrow \tau A_0 \oplus I\longrightarrow 0$$
We have the following exact sequence
$$[\tau M,\tau M]^{1}\longrightarrow [\tau M,\tau A_0\oplus I]^{1}\longrightarrow 0$$
Thus we have $$[M,A_0]^{1}=[A_0,\tau M]=0.$$ Again applying the
functor $\mathrm{Hom}_{k\widetilde{Q}}(A_0,-)$, we have the exact
sequence
$$0\longrightarrow [A_0,N/D_0]\longrightarrow [A_0,\tau M]=0$$
Hence $[A_0,N/D_0]=0.$
\end{proof}

By Lemma \ref{special}, we obtain the following multiplication
theorem between quantum cluster variables in \cite{fanqin}.
\begin{corollary}
Let $M$ and $N$ be indecomposable rigid $kQ$-modules and
$\mathrm{dim}_k\mathrm{Ext}^1_{\mathcal{C}_{\widetilde{Q}}}(M,
N)=1.$ Let
$$N\longrightarrow E\longrightarrow M\longrightarrow N[1]=\tau N$$
and
$$M\longrightarrow D_{0}\oplus A_{0}\oplus I[-1]\longrightarrow N\longrightarrow \tau M$$
be two non-split triangles in $\mathcal{C}_{\widetilde{Q}}$ as
above. Then we have
$$ X_{N}X_M=q^{\frac{1}{2}\Lambda((\widetilde{I}-\widetilde{R})\underline{n},
(\widetilde{I}-\widetilde{R})\underline{m})}X_E+q^{\frac{1}{2}\Lambda((\widetilde{I}-\widetilde{R})\underline{n},
(\widetilde{I}-\widetilde{R})\underline{m})-\frac{1}{2}}X_{D_0\oplus
A_0\oplus I[-1]}.$$
\end{corollary}

Now let $M$ be a $kQ$-module and $P$ be a projective
$k\widetilde{Q}$-module with $[P,M]=[M,I]=1,$ where $I=\nu(P)$, here
$\nu=\mathrm{DHom}_{k\widetilde{Q}}(-,k\widetilde{Q})$ is the
Nakayama functor. It is well-known that $I$ is an injective module
with $\mathrm{soc}I=P/\mathrm{rad}P.$ Fix two nonzero morphisms
$f\in \mathrm{Hom}_{k\widetilde{Q}}(P, M)$ and $g\in
\mathrm{Hom}_{k\widetilde{Q}}(M, I).$ The two morphisms induce  the
following exact sequences
$$\xymatrix{0\ar[r]& P'\ar[r]& P\ar[r]^{f}& M\ar[r]& A\ar[r]& 0}$$
and
$$\xymatrix{0\ar[r]& B\ar[r]& M\ar[r]^{g}& I\ar[r]& I'\ar[r]& 0}.$$
These correspond to two non-split triangles in
$\mathcal{C}_{\widetilde{Q}}$
$$M\longrightarrow E'\longrightarrow P[1]\longrightarrow M[1]$$
and
$$I[-1]\longrightarrow E\longrightarrow M\longrightarrow I,$$
respectively, where $E\simeq B\oplus I'[-1]$ and $E'\simeq A\oplus
P'[1].$

Now we state the second part of our multiplication theorem for
acyclic quantum cluster algebras.
\begin{theorem}\label{exchange2}
With the above notations, assume that $[B,I']=[P',A]=0.$ Then we
have
$$X_{\tau P}X_{M}=q^{\frac{1}{2}\Lambda(\underline{\mathrm{dim}}P/rad P,
-(\widetilde{I}-\widetilde{R})\underline{m})}X_E+q^{\frac{1}{2}\Lambda(\underline{\mathrm{dim}}P/rad
P,
-(\widetilde{I}-\widetilde{R})\underline{m})-\frac{1}{2}}X_{E'}.$$
\end{theorem}
\begin{proof}
\begin{eqnarray}
   && X_{\tau P}X_{M}  \nonumber\\
   &=& X^{\underline{\mathrm{dim}} P/rad P}\sum_{G,H}q^{-\frac{1}{2}\langle H,G\rangle}F^{M}_{GH}X^{\widetilde{B}\underline{h}-(\widetilde{I}-\widetilde{R})\underline{m}}
  \nonumber\\
  &=& \sum_{G,H}q^{-\frac{1}{2}\langle H,G\rangle}F^{M}_{GH}q^{\frac{1}{2}\Lambda(\underline{\mathrm{dim}} P/\mathrm{rad}P,\widetilde{B}\underline{h}-(\widetilde{I}-\widetilde{R})
  \underline{m})}X^{\widetilde{B}\underline{h}-(\widetilde{I}-\widetilde{R})\underline{m}+\underline{\mathrm{dim}}
P/rad P}
\nonumber\\
 &=&q^{\frac{1}{2}\Lambda(\underline{\mathrm{dim}} P/\mathrm{rad}P,-(\widetilde{I}-\widetilde{R})\underline{m})}\sum_{G,H}q^{-\frac{1}{2}\langle
H,G\rangle}q^{-\frac{1}{2}[P,H]}F^{M}_{GH}X^{\widetilde{B}\underline{h}-(\widetilde{I}-\widetilde{R})\underline{m}+\underline{\mathrm{\mathrm{dim}}}
P/\mathrm{rad}P}.\nonumber
\end{eqnarray}
Here we use the following fact
$$\Lambda(\underline{\mathrm{dim}} P/\mathrm{rad}P,\widetilde{B}\underline{h})=(\underline{\mathrm{dim}} P/\mathrm{rad}P)^{tr}\Lambda\widetilde{B}\underline{h}=-
(\underline{\mathrm{dim}} P/\mathrm{rad}P)^{tr}\begin{bmatrix}I_n\\0
\end{bmatrix}\underline{h}=-[P,H].$$ Note that by assumption
$[P,M]=1,$ we have that $[P,H]=0\ or \ 1.$

 We firstly compute the  term $$X_E=X_{B\oplus
I'[-1]}=\sum_{X,Y}q^{-\frac{1}{2}\langle
Y,X\rangle}F^{B}_{XY}X^{\widetilde{B}\underline{y}-(\widetilde{I}-\widetilde{R})\underline{b}+\underline{\mathrm{dim}}
\mathrm{soc}I'}.$$  We have the following diagram
$$
\xymatrix{&0\ar[d]&0\ar[d]\\
&Y\ar@{=}[r]\ar[d]&Y\ar[d]\\
 0\ar[r]&B\ar[r]\ar[d]&M\ar[r]^{\theta}\ar[d]&I\ar[r]&I'\ar[r]&0\\
0\ar[r]&X\ar[r]\ar[d]&G\ar[d]\\
&0&0}
$$
and a short exact sequence
$$0\longrightarrow \mathrm{im}\theta \longrightarrow I\longrightarrow I'\longrightarrow 0.$$
As we assume that $[B,I']=0,$ thus $[H,I']=0.$ Then
\begin{eqnarray}
   && \langle Y,X\rangle-\langle H,G\rangle = \langle H,X\rangle-\langle H,G\rangle =\langle H,X-G\rangle =\langle H,B-M\rangle=-\langle H, \mathrm{im}\theta\rangle.\nonumber
\end{eqnarray}
Applying the functor $[H,-]$ to the above short exact sequence, we
have
$$0\longrightarrow [H, \mathrm{im}\theta]\longrightarrow [H, I]\longrightarrow [H, I']\longrightarrow [H, \mathrm{im}\theta]^{1}\longrightarrow 0$$
Note that $[H, I]=[H, I']=0,$ thus $\langle H,
\mathrm{im}\theta\rangle=0.$ Hence
$$X_E=\sum_{G,H,[P,H]=0}q^{-\frac{1}{2}\langle
H,G\rangle}F^{M}_{GH}X^{\widetilde{B}\underline{h}-(\widetilde{I}-\widetilde{R})\underline{m}+\underline{\mathrm{dim}}
P/\mathrm{rad}P}.$$

 Now compute the  term $$X_{E'}=X_{A\oplus
P'[1]}=\sum_{X,Y}q^{-\frac{1}{2}\langle
Y,X\rangle}F^{A}_{XY}X^{\widetilde{B}\underline{y}-(\widetilde{I}-\widetilde{R})\underline{a}+\underline{\mathrm{dim}}
P'/\mathrm{rad}P'}.$$  We have the following diagram
$$
\xymatrix{&&&0\ar[d]&0\ar[d]\\
&&P\ar[r]\ar@{=}[d]&H\ar[r]\ar[d]&Y\ar[r]\ar[d]&0\\
 0\ar[r]&P'\ar[r]&P\ar[r]^{\theta'}&M\ar[r]\ar[d]&A\ar[r]\ar[d]&0\\
&&&G\ar@{=}[r]\ar[d]&X\ar[d]\\
&&&0&0}
$$
Applying the functor $[P',-]$ to the following short exact sequence
$$0\longrightarrow Y \longrightarrow A\longrightarrow G\longrightarrow 0.$$
 we
have
$$0\longrightarrow [P',Y]\longrightarrow [P', A]\longrightarrow [P', G]\longrightarrow 0$$
As we assume that $[P',A]=0,$ thus $[P',G]=0.$ Then $\langle
P',G\rangle=0.$ Hence we have
\begin{eqnarray}
   && \langle Y,X\rangle-\langle H,G\rangle = \langle Y,G\rangle-\langle H,G\rangle =\langle Y-H,G\rangle =\langle
   A-M,G\rangle\nonumber\\
   &=& \langle P'-P,G\rangle=\langle P',G\rangle=0.\nonumber
\end{eqnarray}
Therefore $$X_E'=\sum_{G,H,[P,H]=1}q^{-\frac{1}{2}\langle
H,G\rangle}F^{M}_{GH}X^{\widetilde{B}\underline{h}-(\widetilde{I}-\widetilde{R})\underline{m}+\underline{\mathrm{dim}}
P/\mathrm{rad}P}.$$ This completes the proof.
\end{proof}

Note that if $M$ is indecomposable rigid object in
$\mathcal{C}_{\widetilde{Q}}$ and $[P,M]=[M,I]=1$, then both
$E=B\oplus I'[-1]$ and $E'=A\oplus P'[1]$ are rigid. Thus the
assumptions
$\mathrm{Hom}_{k\widetilde{Q}}(B,I')=\mathrm{Hom}_{k\widetilde{Q}}(P',A)=0$
in Theorem \ref{exchange2} naturally hold. The quantum cluster
multiplication theorem in \cite{fanqin} deals with this special
case.

\section{Generic bases in  specialized quantum cluster algebras of finite type}
Let $k$ be a finite field with cardinality $|k|=q$ and $m\geq n$ be
two positive integers and $\widetilde{Q}$ an acyclic quiver with
vertex set $\{1,\ldots,m\}$. The full subquiver $Q$ on the vertices
$\{1,\ldots,n\}$ is  the principal part of $\widetilde{Q}$.
 Let $\mathcal{A}_{|k|}(\widetilde{Q})$ be the corresponding specialized
quantum cluster algebra of $Q$ with coefficients. Then the main
theorem in \cite{fanqin} shows that
$\mathcal{A}_{|k|}(\widetilde{Q})$ is the $\mathbb{ZP}$-subalgebra
of $\Fcal$ generated by
$$\{X_M| M\ \text{is  indecomposable rigid $kQ$-mod}\}\cup$$$$
\{X_{\tau P_i},1\leq i\leq n|P_i\ \text{is  indecomposable
projective $k\widetilde{Q}$-mod}\}.$$ Let $i$ be a sink or a source
in $\widetilde{Q}$. We define the reflected quiver
$\sigma_i(\widetilde{Q})$ by reversing all the arrows ending at $i$.
An \emph{admissible sequence of sinks (resp. sources)} is a sequence
$(i_1, \ldots, i_l)$ such that $i_1$ is a sink (resp. source) in
$\widetilde{Q}$ and $i_k$ is a sink (resp source) in
$\sigma_{i_{k-1}}\cdots \sigma_{i_1}(\widetilde{Q})$ for any $k=2,
\ldots, l$. A quiver $\widetilde{Q}'$ is called
\emph{reflection-equivalent}\index{reflection-equivalent} to
$\widetilde{Q}$ if there exists an admissible sequence of sinks or
sources $(i_1, \ldots, i_l)$ such that
$\widetilde{Q}'=\sigma_{i_{l}}\cdots \sigma_{i_1}(\widetilde{Q})$. A
quiver $\widetilde{Q}'$ is called \emph{reachable} from
$\widetilde{Q}$ if $\widetilde{Q}'=\sigma_{i_{l}}\cdots
\sigma_{i_1}(\widetilde{Q})$ where $1\leq i_1,\cdots,i_{l}\leq n$.
Note that mutations can be viewed as generalizations of reflections,
i.e, if $i$ is a sink or a source in a quiver $\widetilde{Q}$,
 then $\mu_i(\widetilde{Q})=\sigma_i(\widetilde{Q})$ where $\mu_i$ denotes the mutation in the direction
 $i$. Thus if $\widetilde{Q}'$ is \emph{reachable} from $\widetilde{Q}$,
 there is a natural canonical isomorphism between
$\mathcal{A}_{|k|}(\widetilde{Q})$ and
$\mathcal{A}_{|k|}(\widetilde{Q}'),$ denoted by
$$\Phi_{i}: \mathcal{A}_{|k|}(\widetilde{Q})\rightarrow
\mathcal{A}_{|k|}(\widetilde{Q}').$$

Let $\Sigma_i^+:\ \mathrm{mod}(\widetilde{Q}) \longrightarrow \
\mathrm{mod}(\widetilde{Q}')$ be the standard BGP-reflection functor
and $R_i^+:\mathcal C_{\widetilde{Q}} \longrightarrow \mathcal
C_{\widetilde{Q}'}$ be the
        extended BGP-reflection functor defined by \cite{Zhu}:
        $$R_i^+:\left\{\begin{array}{rcll}
            X & \mapsto & \Sigma_i^+(X), & \textrm{ if }X \not \simeq S_i \textrm{ is a $kQ$-module,}\\
            S_i & \mapsto & P_i[1], \\
            P_j[1] & \mapsto & P_j[1], & \textrm{ if }j \neq i,\\
            P_i[1] & \mapsto & S_i.
        \end{array}\right.$$
    By Rupel \cite{rupel}, the following holds:
\begin{theorem}\cite[Theorem 2.4, Lemma 5.6]{rupel}\label{ref}
For any $ X_M^{\widetilde{Q}}\in\mathcal{A}_{|k|}(\widetilde{Q})$,
we have $\Phi_{i}(X_M^{\widetilde{Q}})=X_{R_i^+M}^{\widetilde{Q}'}.$
\end{theorem}

\begin{definition}[\cite{CK1}]
            Let $Q$ be an acyclic quiver with associated matrix $B$. $Q$ will be called \emph{graded}\index{graded} if
            there exists a linear form $\epsilon$ on $\mathbb{Z}^{n}$ such that $\epsilon(B \alpha_i)<0$ for any $1\leq i \leq n$
            where $\alpha_i$ still denotes the $i$-th
            vector of the canonical basis of $\mathbb{Z}^{n}$.
\end{definition}

If $Q$ is a graded quiver, then it is proved in \cite{CK1} that we
can endow the cluster algebra of $Q$ with a grading. Namely, the
results are the following:

For any Laurent polynomial $P$ in the variables $X_i$, the
 $supp(P)$ of $P$ is defined as the set of points
$\lambda=(\lambda_i,1\leq i \leq n)$ of $\mathbb{Z}^{n}$ such that
the $\lambda$-component, that is, the coefficient of $\prod_{1\leq i
\leq n} X_i^{\lambda_i}$ in $P$ is nonzero. For any $\lambda$ in
$\mathbb{Z}^{n}$, let $C_\lambda$ be the convex cone with vertex
$\lambda$ and edge vectors generated by the $B\alpha_i$ for any
$1\leq i \leq n$. Then we have the following two propositions as the
quantum versions of Proposition 5 and Proposition 7 in \cite{CK1}
respectively.
\begin{proposition}\label{prop:supportconeCK1}
            Let $Q$ be a graded acyclic quiver with no multiple arrows and
           $M=M_0 \oplus P_M[1]$ with $M_0$ is $kQ$-module and $P_M$ projective $k\widetilde{Q}$-module. Then,
            $supp(X_{M_0\oplus P_M[1]})$
            is in $C_{\lambda_M}$ with
            $\lambda_M:=(-\langle\alpha_i,\underline{dim} M_0\rangle+\langle\underline{dim} P_M, \alpha_i\rangle)
            _{1\leq i \leq n}$.  Moreover, the $\lambda_M$-component of $X_{M_0\oplus P_M[1]}$
          is some nonzero monomials in $\{|k|^{\pm\frac{1}{2}},X^{\pm 1
            }_{n+1},\cdots,X^{\pm 1
            }_{m}\}$.
\end{proposition}

\begin{proposition}\label{prop:graduationCK1}
            Let $Q$ be a graded acyclic quiver with no multiple arrows. For any $m \in \mathbb{Z}$, set
            $$F_m=\left( \bigoplus_{\epsilon(\nu) \leq m} \mathbb{ZP}\prod_{1\leq i \leq n}u_i^{\nu_i}\right) \cap \Acal_{|k|}(\widetilde{Q}),$$
            then the set $(F_m)_{m \in \mathbb{Z}}$ defines a $\mathbb{Z}$-filtration of $\Acal_{|k|}(\widetilde{Q})$.
\end{proposition}

For any $\underline{d}\in \mathbb{Z}^{n},$ define
$\underline{d}^{+}=(d^+_i)_{1\leq i \leq n}$ such that $d^+_i=d_i$
if $d_i>0$ and $d_i^+=0$ if $d_i\leq 0$ for any $1\leq i \leq n.$
Dually, we set $\underline{d}^-=\underline{d}^+-\underline{d}.$ The
following proposition \ref{10} can be viewed as the categorification
of   \cite[Theorem 7.3]{berzel}.

\begin{proposition}\label{10}
Let $\widetilde{Q}$ be an acyclic quiver. Then the set
$\{\prod_{i=1}^{n}X^{d^+_i}_{S_i}X^{d^-_i}_{P_i[1]}\mid
(d_1,\cdots,d_n)\in \mathbb{Z}^n\}$ is a $\mathbb{ZP}$-basis of
$\Acal_{|k|}(\widetilde{Q})$.
\end{proposition}
\begin{proof}
For any $1\leq i\leq n,$ it is easy to check that the following set
is a cluster $$\{X_{\tau P_{1}},\cdots,X_{\tau P_{i-1}},
X_{S_{i}},X_{\tau P_{i+1}},\cdots, X_{\tau P_{n}}\}$$ obtained by
the mutation in direction $i$ of the cluster $$\{X_{\tau
P_{1}},\cdots,X_{\tau P_{i-1}}, X_{\tau P_{i}},X_{\tau
P_{i+1}},\cdots, X_{\tau P_{n}}\}.$$ Then the proposition
immediately follows from \cite[Theorem 7.3]{berzel} and
\cite[Theorem 5.4.3]{fanqin}.
\end{proof}

The main result is the following theorem showing the
$\mathbb{ZP}$-basis in a specialized quantum cluster algebra of
finite type. The basis is the good basis in a cluster algebra of
finite type in \cite{CK1} by specializing $q$ and coefficients to
$1$ and the existence of Hall polynomials for representation-finite
algebras \cite{ringel:2}.

\begin{theorem}\label{12}
Let $Q$ be a simple-laced Dynkin quiver with $Q_0=\{1,2,\cdots, n\}$
and  $\widetilde{Q}$  reachable from $\widetilde{Q'}$ for any graded
quiver $Q'$.  Then the set $\mathcal{B}(Q):=\{X_{M}|M=M_0 \oplus
P_M[1]\ \text{with}\ M_0\ \text{is}\   kQ\text{-module},\ P_M\
\text{a direct sum of projective}\ k\widetilde{Q}\text{-modules}\
P_i(1\leq i\leq n),\
 M\ \text{rigid}$
$\text{object in}\ \mathcal C_{\widetilde{Q}}\}$ is a
$\mathbb{ZP}$-basis of $\Acal_{|k|}(\widetilde{Q})$.
\end{theorem}
\begin{proof}
Assume that $\sigma_{i_{l}}\cdots
\sigma_{i_1}(\widetilde{Q'})=\widetilde{Q}(1\leq
i_1,\cdots,i_{l}\leq n)$ . For any $X^{\widetilde{Q}'}_{M}\in
\mathcal{B}(Q')$  with dimension vector
$\mathrm{\underline{dim}}M=\underline{m}=(m_1,\cdots,m_n)\in
\mathbb{Z}^n$, we know that $X^{\widetilde{Q}'}_{M}\in
\Acal_{|k|}(\widetilde{Q}')$. Then by Proposition \ref{10} we have
$$X^{\widetilde{Q}'}_{M}=b_{\underline{m}}\prod_{i=1}^{n}(X_{S_i}^{\widetilde{Q}'})^{m^+_i}(X^{\widetilde{Q}'}_{P_i[1]})^{m^-_i}
+\sum_{\epsilon(\underline{l})<
\epsilon(\underline{m})}b_{\underline{l}}\prod_{i=1}^{n}(X_{S_i}^{\widetilde{Q}'})^{l^+_i}(X^{\widetilde{Q}'}_{P_i[1]})^{l^-_i}$$
where $\underline{l}=(l^+_i-l^-_i)_{i\in Q_0}$, $b_{\underline{m}}$
and $b_{\underline{l}}\in \mathbb{ZP}$. As $Q'$ is a graded quiver,
then by Proposition \ref{prop:supportconeCK1}, Proposition
\ref{prop:graduationCK1}, it follows that $b_{\underline{m}}$ must
be some nonzero monomial in $\{q^{\pm\frac{1}{2}},X^{\pm 1
            }_{n+1},\cdots,X^{\pm 1
            }_{m}\}$. Therefore,  $\mathcal{B}(Q')$ is a $\mathbb{ZP}$-basis of
$\Acal_{|k|}(\widetilde{Q}')$. There is a natural isomorphism:
$\Phi_{i_{l}}\cdots \Phi_{i_1}:
\mathcal{A}_{|k|}(\widetilde{Q}')\rightarrow
\mathcal{A}_{|k|}(\widetilde{Q})$. By Theorem \ref{ref}, we obtain
that
$$\Phi_{i_{l}}\cdots \Phi_{i_1}(X^{\widetilde{Q}'}_M)=X^{\widetilde{Q}}_{R^+_{i_{l}}\cdots
R^+_{i_1}(M)}.$$ Hence, $\mathcal{B}(Q)$ is a $\mathbb{ZP}$-basis of
$\Acal_{|k|}(\widetilde{Q})$.

\end{proof}

By the existence of Hall polynomials for representation-finite
algebras \cite{ringel:2}, we have the following corollary.
\begin{corollary}
With the above notation, the set $\mathcal{B}(Q)$ is a
$\mathbb{ZP}$-basis of $\Acal_{q}(\widetilde{Q})$.
\end{corollary}

\section{Generic bases in  specialized quantum cluster algebras of  affine type}
Throughout this section, we assume that $Q$ is a tame quiver with
trivial valuation. A \emph{tame quiver}\index{affine quiver} is an
acyclic quiver whose underlying diagram in an extended Dynkin
diagram. We recall some facts about representation theory of tame
quivers (for example, refer to
\cite{dlab}\cite{CB:lectures}\cite{ringel:1099} for more details).
Let $k$ be a finite field with $|k|=q$. The category $rep(kQ)$ of
finite-dimensional representations can be identified with the
category of mod-$kQ$ of finite-dimensional modules over the path
algebra $kQ.$ For any $kQ$-representation $M$ and $i\in Q_0$, we
denote by $(M)_i$ the $k$-space at $i$. It is well-known that
indecomposable $kQ$-module contains (up to isomorphism) three
families (by the Auslander-Reiten quiver): the component of
indecomposable regular modules $\mathcal R(Q)$, the component of the
preprojective modules $\mathcal P(Q)$ and the component of the
preinjective modules $\mathcal I(Q)$. If $P \in \mathcal P(Q)$, $I
\in \mathcal I(Q)$ and $R \in \mathcal R(Q)$, then
    $$\Hom_{kQ}(R,P) \simeq \Hom_{kQ}(I,R) \simeq \Hom_{kQ}(I,P)=0,$$
    and
    $$\Ext^1_{kQ}(P,R)\simeq \Ext^1_{kQ}(R,I)\simeq \Ext^1_{kQ}(P,I)=0.$$
    If $M$ and $N$ are two regular indecomposable modules in different tubes, then
    $$\Hom_{kQ}(M,N)=0 \textrm{ and } \Ext^1_{kQ}(M,N)=0.$$

The Auslander-Reiten quiver of $\mathcal{R}(Q)$ consists of tubes.
An indecomposable regular module $R$ is regular simple if it
contains no nontrivial regular submodule, and call it homogeneous if
$\tau_{Q} R\cong R$. Any regular  module at the bottom of a tube is
regular simple module. If it is homogeneous, then the tube is a
homogeneous tube, otherwise, called a non-homogeneous tube.  For a
regular module $R$, its degree is the index $[\mathrm{End}_{kQ}(R) :
k]$. There are at most $t\leq 3$ non-homogeneous tubes for $Q.$ We
denote these non-homogeneous tubes by $\mathcal{T}_1, \cdots,
\mathcal{T}_t$. Let $r_i$ be the rank of $\mathcal{T}_i$ and the
regular simple modules in $\mathcal{T}_i$ be $E^{(i)}_{1}, \cdots
E^{(i)}_{r_i}$ such that $\tau_{Q} E^{(i)}_2=E^{(i)}_1, \cdots,
\tau_{Q} E^{(i)}_1=E^{(i)}_{r_i}$ for $i=1, \cdots, t$. If we
restrict the discussion to one tube, we will omit the index $i$ for
convenience. There are $q+1-t$ homogeneous tubes which regular
simples at the bottoms are of degree $1$. Given a regular simple
module $E$ in a tube, $E[i]$ is the indecomposable regular module
with quasi-socle $E$ and quasi-length $i$ for any $i\in \mathbb{N}$.
The minimal imaginary root of $Q$ is denoted by
$\delta=(\delta_i)_{i\in Q_0}.$ Note that the regular simple module
of degree $1$ at the bottom in a homogeneous tube is of dimension
vector $\delta$. We now prove that the quantum Caldero-Chapoton map
does not depend on the modules in the homogeneous tube with
dimension vector $\delta$.
\begin{lemma}\label{independent}
Let $\lambda$ and $\lambda'$ be in $k$ such that $E(\lambda)$ and
$E(\lambda')$ are two regular simple modules of dimension vector
$\delta$. Then $X_{E(\lambda)}=X_{E(\lambda')}.$
\end{lemma}
\begin{proof}
We only need to prove
$|Gr_{\underline{e}}(E(\lambda))|=|Gr_{\underline{e}}(E(\lambda'))|.$
 Let $e$ be a vertex in $Q_0$ such that
$\delta_e=1.$  If $e$ is a sink and $Q$ is a quiver with the
underlying graph not of type $\widetilde{A}_n$,  then there is
unique edge $\alpha\in Q_1$ such that $t(\alpha)=e.$  It is easy to
check $\delta_{s(\alpha)}=2.$ Let $P_e$ be the simple projective
module corresponding $e$ and $I$ be the indecomposable preinjective
module of dimension vector $\delta-\mathrm{\underline{dim}}P_e$.
Then $\mathrm{dim}_{k}\mathrm{Ext}^1_{k Q}(I, P_e)=2$. Given any
$\varepsilon\in \mathrm{Ext}^1_{k Q}(I, P_e),$ we have a short exact
sequence whose equivalence class is $\varepsilon$
$$
\xymatrix{\varepsilon:\quad 0\ar[r]& P_e\ar[rr]^{\left(%
\begin{array}{c}
  1 \\
  0 \\
\end{array}%
\right)}&&E_{\varepsilon}\ar[rr]^{\left(%
\begin{array}{cc}
  0 & 1 \\
\end{array}%
\right)}&&I\ar[r]&0}
$$
where $(E_{\varepsilon})_{i}=(P_e)_i\oplus (I)_i$ for any $i\in
Q_0$, $(E_{\varepsilon})_{\beta}=(I)_{\beta}$ for $\beta\neq \alpha$
and $(M_{\epsilon})_{\alpha}$ is
$$
(E_{\varepsilon})_{\alpha}=\left(%
\begin{array}{cc}
  0 & m(\varepsilon, \alpha) \\
  0 & 0 \\
\end{array}%
\right)
$$ where $m(\varepsilon, \alpha)\in \mathrm{Hom}_{k}((I)_{s(\alpha)}, (P_e)_e).$
 Any regular
simple $k Q$-module $E$ of dimension vector $\delta$ satisfies that
$E_{\alpha}$ is as follows
$$
\xymatrix{(E)_{s(\alpha)}=k^2\ar[rr]^{\left(%
\begin{array}{cc}
  1 & \lambda \\
\end{array}%
\right)}&&(E)_{e}=k}
$$
where $\lambda\in k^*\setminus\{1\}.$   Let $P$ be an indecomposable
projective module such that $P\subseteq M(\lambda)$ and $d_{P,
e}=0$. Then $P$ is also a submodule of $E(\lambda')$ and $d_{P,
s(\alpha)}=0$. Let $P$ be an indecomposable projective module such
that $P\subseteq E(\lambda)$ and $d_{P, e}=1$. Assume that
$P_{\alpha}$ is $\xymatrix{k\ar[r]^{a+b\lambda}&k}$, then there
exists $P'\in Gr_{\underline{e}}(E(\lambda'))$  such that $P'\cong
P$ and $P'_{\alpha}$ is $\xymatrix{k\ar[r]^{a+b\lambda'}&k}$. Since
$\tau E(\lambda)=E(\lambda)$, we know $\tau^{-i}P$ and $\tau^{-1}P'$
are the submodules of $E(\lambda)$ and $E(\lambda')$ for any $i\in
\mathbb{N}$, respectively. Hence, any preprojective submodule $X$ of
$E(\lambda)$ corresponds to a preprojective submodule $X'$ of
$E(\lambda')$ such that $X\cong X'$.  Let $Q$ be of type
$\widetilde{A}_n$. Then there are two adjacent edge $\alpha$ and
$\beta$. Any regular simple module $E$ of dimension vector $\delta$
satisfies that $E_{\alpha}$ is as follows:
$$
\xymatrix{k\ar[r]^1&k&k\ar[l]_{\lambda}}$$ for some $\lambda\in
k^{*}.$ The discussion is similar as above.  If $e$ is a source, the
discussion is also similar. Therefore, there is a homeomorphism
between $Gr_{\underline{e}}(E(\lambda))$ and
$Gr_{\underline{e}}(E(\lambda'))$ for any dimension vector
$\underline{e}.$
\end{proof}

Now let $Q$ be a graded tame quiver and $E_{1}, \cdots E_{r}$ be
regular simple modules in a nonhomogeneous tube with rank $r$ such
that $\tau_{Q} E_2=E_1, \cdots, \tau_{Q} E_1=E_{r}$. Let
$\widetilde{Q}$ be a quiver obtained from $Q$ by adding frozen
vertices $\{n+1,\cdots,2n\}$ and arrows $n+i\longrightarrow i$ for
any $1\leq i\leq n.$ Then we have the following result.
\begin{lemma}\label{k2}
$$X_{E_i[r-1]}X_{E_{i-1}}=q^{\frac{1}{2}\Lambda((\widetilde{I}-\widetilde{R})\underline{e_i[r-1]},
(\widetilde{I}-\widetilde{R})\underline{e_{i-1}})}X_{E_i[r]}+q^{\frac{1}{2}\Lambda((\widetilde{I}-\widetilde{R})\underline{e_i[r-1]},
(\widetilde{I}-\widetilde{R})\underline{e_{i-1}})-\frac{1}{2}}X_{E_i[r-2]\oplus
I[-1]}$$ where $I$ is an injective $k\widetilde{Q}$-module
associated to frozen vertices.
\end{lemma}
\begin{proof}
We have the following exact sequences
$$0\longrightarrow E_i[r-1]\longrightarrow E_i[r]\longrightarrow E_{i-1}\longrightarrow 0$$
$$0\longrightarrow E_i[r-2]\longrightarrow E_i[r-1]\longrightarrow \tau_{\widetilde{Q}}E_{i-1}\longrightarrow I\longrightarrow 0.$$
Hence the proof follows from Theorem \ref{multi-formula}.
\end{proof}

Define the set
$$
\textbf{D}(Q)=\{\underline{d}\in \mathbb{N}^ {Q_0}\mid  \exists
\mbox{ a regular rigid module}\ R\   \mbox{and a regular simple
module}\ E$$$$ \mbox{\ with dimension vector }\delta \mbox{ such
that } \mathrm{\underline{dim}}((E^{\oplus n}\oplus
R)=\underline{d}\}.
$$
 Set $\textbf{E}(Q)=\{\underline{d}\in
\mathbb{Z}^{Q_0}\mid \exists M=M_0 \oplus P_M[1]\ \text{with}\ M_0\
\text{is}\  kQ\text{-module},\  P_M\ \text{projective}\
k\widetilde{Q}\text{-module},$
 $M\ \text{rigid object in}\ \mathcal C_{\widetilde{Q}}\  \text{with}\ \mathrm{\underline{dim}} M=\underline{d}\}$. By the main theorem in
\cite{DXX}, we have that $\mathbb{Z}^{Q_0}$ is the disjoint union of
$\textbf{D}(Q)$ and $\textbf{E}(Q)$. We make an assignment, i.e., a
map $$\phi: \mathbb{Z}^{Q_0}\rightarrow
\mathrm{obj}(\mathcal{C}_{\widetilde{Q}})$$ and set
$$X_{\phi(\underline{d})}:=(X_{E})^{n}X_{R}$$ if $\underline{d}\in\textbf{D}(Q)$ and $|Q_0|>2;$

$$
X_{\phi(\underline{d})}:=X^{n}_{\delta}
$$ for some $\delta$ in a homogeneous tube of degree $1$ if $\underline{d}\in\textbf{D}(Q)$ and $Q$ is the Kronecker
quiver;
$$
X_{\phi(\underline{d})}:=X_{M}
$$ if $\underline{d}\in\textbf{E}(Q)$. It is clear that the above assignment is not
unique. For simplicity and no confusions, we omit $\phi$ in the
notation $X_{\phi(\underline{d})}$.

\begin{theorem}\label{affine}
Let k be a finite field with $|k|=q\neq 2.$ Let $Q$ be a tame quiver
with $Q_0=\{1,2,\cdots, n\}$ and $\widetilde{Q}$ reachable from
$\widetilde{Q'}$ for any graded quiver $Q'$. Then the set
$$\mathcal{B}(Q):=\{X_{\underline{d}}|\underline{d}\in \mathbb{Z}^{Q_0}\}$$
is a $\mathbb{QP}$-basis of
$\Acal_{|k|}(\widetilde{Q})\otimes_{\mathbb{ZP}} \mathbb{QP}$.
\end{theorem}
\begin{proof}
Assume that $e$ is a sink in $Q'$ and $P_e$ projective $kQ$-module
associated to $e.$ Let $I$ be an indecomposable preinjective
$kQ$-module with dimension vector $\delta-\mathrm{\underline{dim}}
P_e$. By Lemma \ref{independent} and  Theorem \ref{hall multi}, we
have
$$q^{2}X_{P_e}X_{I}=q^{-\frac{1}{2}\Lambda((\widetilde{I}-\widetilde{R})\underline{i},
(\widetilde{I}-\widetilde{R})\underline{p_e})} (X_{P_e\oplus
I}+(q+1-t)\varepsilon_{IP_e}^{E}X_E+\sum_{i=1}^{t}\varepsilon_{IP_e}^{E_1[r_i]}X_{E_1[r_i]}).$$
and
$$X_{I}X_{P_e}=q^{-\frac{1}{2}\Lambda((\widetilde{I}-\widetilde{R})\underline{p_e},
(\widetilde{I}-\widetilde{R})\underline{i})} X_{P_e\oplus I}.$$ It
follows that
$$q^{\frac{1}{2}\Lambda((\widetilde{I}-\widetilde{R})\underline{i},
(\widetilde{I}-\widetilde{R})\underline{p_e})+2}
X_{P_e}X_{I}-q^{\frac{1}{2}\Lambda((\widetilde{I}-\widetilde{R})\underline{p_e},
(\widetilde{I}-\widetilde{R})\underline{i})}X_{I}X_{P_e}=(q+1-t)\varepsilon_{IP_e}^{E}X_E+\sum_{i=1}^{t}\varepsilon_{IP_e}^{E_1[r_i]}X_{E_1[r_i]}$$
where, by definition,
$\varepsilon_{IP_e}^{E}=\varepsilon_{IP_e}^{E_1[r_i]}=q-1.$ By Lemma
\ref{k2}, $X_{E_1[r_i]}$ is in
$\Acal_{|k|}(\widetilde{Q}')\otimes_{\mathbb{ZP}} \mathbb{QP}$ for
$1\leq i\leq t$, thus
 $X_E$  is in $\Acal_{|k|}(\widetilde{Q}')\otimes_{\mathbb{ZP}} \mathbb{QP}$.
It follows that for any $\underline{m}=(m_1,\cdots,m_n)\in
\mathbb{Z}^n$, $X_{\underline{m}}\in\mathcal{B}(Q')$.  Then by
Proposition \ref{10} we have
$$X^{\widetilde{Q}'}_{\underline{m}}=b_{\underline{m}}\prod_{i=1}^{n}(X_{S_i}^{\widetilde{Q}'})^{m^+_i}(X^{\widetilde{Q}'}_{P_i[1]})^{m^-_i}
+\sum_{\epsilon(\underline{l})<
\epsilon(\underline{m})}b_{\underline{l}}\prod_{i=1}^{n}(X_{S_i}^{\widetilde{Q}'})^{l^+_i}(X^{\widetilde{Q}'}_{P_i[1]})^{l^-_i}$$
where  $b_{\underline{m}},b_{\underline{l}}\in \mathbb{ZP}$. As $Q'$
is a graded quiver, then by Proposition \ref{prop:supportconeCK1},
Proposition \ref{prop:graduationCK1}, it follows that
$b_{\underline{m}}$ must be some nonzero monomial in
$\{q^{\pm\frac{1}{2}},X^{\pm 1
            }_{n+1},\cdots,X^{\pm 1
            }_{m}\}$. Therefore,  $\mathcal{B}(Q')$ is a $\mathbb{QP}$-basis of
$\Acal_{|k|}(\widetilde{Q}')\otimes_{\mathbb{ZP}} \mathbb{QP}$. By
Theorem \ref{ref}, we obtain that $\mathcal{B}(Q)$ is a
$\mathbb{QP}$-basis of
$\Acal_{|k|}(\widetilde{Q})\otimes_{\mathbb{ZP}} \mathbb{QP}$.
\end{proof}

By \cite[Proposition 5]{CR}, the quiver Grassmannians
$\mathrm{Gr}_{\underline{e}}(M)$ of a $kQ$-module $M$ are some
polynomials in $\mathbb{Z}[q].$ Then by specializing $q$ and
coefficients to $1$, the bases in Theorem \ref{affine} induces the
integral bases in affine cluster algebras (\cite{DXX}\cite{Dup}). In
the same way, we have the following corollary.
\begin{corollary}
With the above notation, the set $\mathcal{B}(Q)$ is a
$\mathbb{QP}$-basis of
$\Acal_{q}(\widetilde{Q})\otimes_{\mathbb{ZP}} \mathbb{QP}$.
\end{corollary}

\subsection{An example} Let $Q$ be the tame quiver of type
$\widetilde{A}^{(1)}_{1}$ as follows
$$\xymatrix{1\ar @<2pt>[r] \ar@<-2pt>[r]&  2}$$

It is well known that the  regular indecomposable modules decomposes
into a direct sum of homogeneous tubes  indexed by the projective
line $\mathbb{P}^1(k)$. We denote the  regular indecomposable
modules in the homogeneous tube for $p\in \mathbb{P}^1(k)$ of degree
$1$ by $R_p(n)$ where $n\in\mathbb{N}$ and
$\underline{\mathrm{dim}}R_p(n)=(n, n)$.

We consider the following  ice quiver $\widetilde{Q}$ with frozen
vertices $3$ and $4$:
$$
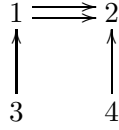
\xymatrix{1\ar @<2pt>[r] \ar@<-2pt>[r]&  2\\
3\ar[u]&  4\ar[u]}$$ Thus we have
$$\widetilde{R}=\left(\begin{array}{cc} 0 & 0\\
2& 0\\
0& 0\\
0 & 0\end{array}\right),\ \widetilde{I}=\left(\begin{array}{cc} 1 & 0\\
0& 1\\
0& 0\\
0 & 0\end{array}\right),\ \widetilde{B}=\left(\begin{array}{cc} 0 & 2\\
-2& 0\\
1& 0\\
0 & 1\end{array}\right).$$

An easy calculation shows that the following antisymmetric $4\times
4$ integer
matrix $$\Lambda=\left(\begin{array}{cccc} 0 & 0& -1& 0\\
0 & 0& 0& -1\\
1 & 0& 0& -2\\
0 & 1& 2& 0\end{array}\right)$$ satisfying
\begin{align}\label{eq:simply_laced_compatible}
\Lambda(-\widetilde{B})=\widetilde{I}:=\begin{bmatrix}I_2\\0
\end{bmatrix},
\end{align}
where $I_2$ is the identity matrix of size $2\times 2$. Then we have
the following result.

\begin{lemma}\label{kro}
Let $R_{p}(1)$ be the indecomposable regular module of degree $1$ as
above. Then
$$X_{R_p(1)}=X_{S_1}X_{S_2}-q^{-\frac{1}{2}}X_1X_2X_4.$$
\end{lemma}
\begin{proof}
By definition, we have
$$X_{S_1}=X^{(-1,0,1,0)}+X^{(-1,2,0,0)};$$
$$X_{S_2}=X^{(0,-1,0,0)}+X^{(2,-1,0,1)};$$
$$X_{R_p(1)}=X^{(1,-1,1,1)}+X^{(-1,1,0,0)}+X^{(-1,-1,1,0)}.$$
Hence the lemma follows from a direct calculation.
\end{proof}
By Lemma \ref{1}, the expression of $X_{R_p(1)}$ is independent of
the choice of  $p\in \mathbb{P}^1_k$ of degree 1. Hence, we set
$$
X_{\delta}:= X_{R_p(1)}.
$$

\begin{remark}\label{kro2}

\begin{enumerate}
  \item By Lemma \ref{kro}, we know that $X_{\delta}$ belongs to $\Acal_{|k|}(\widetilde{Q}).$
 \item By the following Theorem \ref{affine}, $\mathcal{B}(Q)$ is a
$\mathbb{ZP}-$basis in the quantum cluster algebra
$\Acal_{|k|}(\widetilde{Q}).$ Moreover, if specializing $q$ and
coefficients to $1$, $\mathcal{B}(Q)$ is exactly the generic basis
in the sense of \cite{Dup}.
\end{enumerate}
\end{remark}
Note that there is an alternative choice of $(\Lambda,
\widetilde{B})$, i.e.,
$\widetilde{Q}=Q$ and set $\Lambda=\left(\begin{array}{cc} 0 & 1\\
-1 &
0\end{array}\right)$ and $\widetilde{B}=\left(\begin{array}{cc} 0 & 2\\
-2 & 0\end{array}\right)$. Then we have $
\Lambda(-\widetilde{B})=\left(
                          \begin{array}{cc}
                            2 & 0 \\
                            0 & 2 \\
                          \end{array}
                        \right).
$ Hence, one should consider the category of $KQ$-representations
for a field K with $|K|=q^2.$ In this way, we obtain a quantum
cluster algebra of Kronecker type without coefficients. The
multiplication and the bar-invariant bases in this algebra have been
thoroughly studied in \cite{DX}. Moreover, under the specialization
$q=1$, the bases in \cite{DX} induce the canonical basis,
semicanonical basis and generic basis of the cluster algebra of the
Kronecker quiver in the sense of \cite{sherzel},\cite{calzel} and
\cite{Dup}, respectively.
\section{Remark on the difference property of affine quantum cluster
algebras}In this section, we prove that bases in Theorem
\ref{affine} are $\mathbb{ZP}$-bases for quantum cluster algebras of
type $\widetilde{A}$ and $\widetilde{D}$. This refines Theorem
\ref{affine}.
\subsection{The case in  type $\widetilde{A}_{r,s}$}
Let $Q$ be a quiver of type $\widetilde{A}_{r, s}$ as follows.
$$
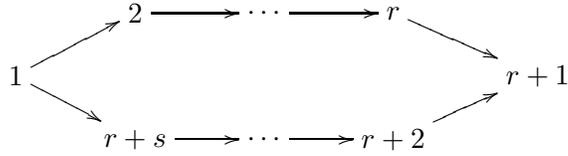
\xymatrix@R=0.8pc{&2\ar[r]&\cdots\ar[r]&r\ar[rd]&\\
1\ar[ur]\ar[dr]&&&&r+1\\
&r+s\ar[r]&\cdots\ar[r]&r+2\ar[ur]&}
$$
There are two non-homogeneous tubes (denoted by $\mathcal{T}_0,
\mathcal{T}_\infty$) in the set of indecomposable regular modules.
The minimal imaginary root of $Q$ is $\delta=(1,1, \cdots, 1).$ Let
$E(\lambda)$ be the regular simple module in the homogeneous tube
for $\lambda\in \mathbb{P}^1$ of degree $1$. The regular simple
modules in $\mathcal{T}_0$ are
$$ \xymatrix@R=0.5pc{&&k\ar[r]&\cdots\ar[r]&k\ar[rd]^{1}&\\
E^{(0)}_{t}=S_{r+t+1} \mbox{ for }1\leq t\leq s-1,\quad E^{(0)}_s: &k\ar[ur]\ar[dr]&&&&k\\
&&0\ar[r]&\cdots\ar[r]&0\ar[ur]^{0}&}  $$ The regular module
$E^{(0)}_1[s]$ has the form as follows
$$
\xymatrix@R=0.8pc{&k\ar[r]&\cdots\ar[r]&k\ar[rd]^{1}&\\
k\ar[ur]\ar[dr]&&&&k\\
&k\ar[r]&\cdots\ar[r]&k\ar[ur]^{0}&}
$$
By \cite{DXX}, we have
$$
|Gr_{\underline{e}}(E^{(0)}_1[s])|=|Gr_{\underline{e}}(E(\lambda))|+|Gr_{\underline{e}-\mathrm{\underline{dim}}S_{r+2}}(E^{(0)}_2[s-2])|
\ \ \ (\star)$$ The follow lemma can be viewed as the difference
property of quantum cluster algebra of $\widetilde{A}_{r, s}$.
\begin{proposition}\label{k1}
$$X_{E^{(0)}_1[s]}=X_{E(\lambda)}+q^{\frac{1}{2}}X_{E^{(0)}_2[s-2]}$$
\end{proposition}
\begin{proof}
By the above equation $(\star)$, we have
\begin{eqnarray}
   X_{E^{(0)}_1[s]} &=& \sum_{\underline{v}} |\mathrm{Gr}_{\underline{v}} E^{(0)}_1[s]|q^{-\frac{1}{2}
\langle\underline{v},\underline{\delta}-\underline{v}\rangle}X^{\widetilde{B}\underline{v}-(\widetilde{I}-\widetilde{R})\underline{\delta}} \nonumber\\
&=& \sum_{\underline{v}} |\mathrm{Gr}_{\underline{v}}
E(\lambda)|q^{-\frac{1}{2}
\langle\underline{v},\underline{\delta}-\underline{v}\rangle}X^{\widetilde{B}\underline{v}-(\widetilde{I}-\widetilde{R})\underline{\delta}}+
\sum_{\underline{v}\neq \delta}
|\mathrm{Gr}_{\underline{v}-\underline{s_{r+2}}}
E^{(0)}_2[s-2]|q^{-\frac{1}{2}
\langle\underline{v},\underline{\delta}-\underline{v}\rangle}X^{\widetilde{B}\underline{v}-(\widetilde{I}-\widetilde{R})\underline{\delta}}\nonumber\\
  &=& X_{E(\lambda)}+q^{\frac{1}{2}}\sum_{\underline{v}'} |\mathrm{Gr}_{\underline{v}'} E^{(0)}_2[s-2]|q^{-\frac{1}{2}
\langle\underline{v}',\underline{e^{(0)}_2[s-2]}-\underline{v}'\rangle}X^{\widetilde{B}\underline{v}'-(\widetilde{I}-\widetilde{R})\underline{e^{(0)}_2[s-2]}}\nonumber\\
&=& X_{E(\lambda)}+q^{\frac{1}{2}}X_{E^{(0)}_2[s-2]}.\nonumber
\end{eqnarray}
Here we use the  facts that
\begin{eqnarray}
    \langle\underline{v}',\underline{e^{(0)}_2[s-2]}-\underline{v}'\rangle &=& \langle\underline{v}-\underline{s_{r+2}},
    \delta-\underline{v}-\tau_{Q}\underline{s_{r+2}}\rangle \nonumber\\
&=&
\langle\underline{v},\underline{\delta}-\underline{v}\rangle-\langle\underline{v},\tau_{Q}\underline{s_{r+2}}\rangle
-\langle\underline{s_{r+2}},\delta\rangle+\langle\underline{s_{r+2}},\tau_{Q}\underline{s_{r+2}}\rangle+\langle\underline{s_{r+2}},\underline{v}\rangle\nonumber\\
&=& \langle\underline{v},\underline{\delta}-\underline{v}\rangle
-\langle\underline{s_{r+2}},\delta\rangle+\langle\underline{s_{r+2}},\tau_{Q}\underline{s_{r+2}}\rangle+2\langle\underline{s_{r+2}},\underline{v}\rangle\nonumber\\
   &=& \langle\underline{v},\underline{\delta}-\underline{v}\rangle-0-1+2\nonumber\\
   &=& \langle\underline{v},\underline{\delta}-\underline{v}\rangle+1.\nonumber
\end{eqnarray}
And
\begin{eqnarray}
   && \widetilde{B}\underline{v}'-(\widetilde{I}-\widetilde{R})\underline{e^{(0)}_2[s-2]} \nonumber\\
   &=&\widetilde{B}(\underline{v}-\underline{s_{r+2}})-(\widetilde{I}-\widetilde{R})(\delta-\underline{s_{r+2}}-\tau_{Q}\underline{s_{r+2}})\nonumber\\
  &=& \widetilde{B}\underline{v}-(\widetilde{I}-\widetilde{R})\underline{\delta}-\widetilde{B}\underline{s_{r+2}}+
  (\widetilde{I}-\widetilde{R})(\underline{s_{r+2}}+\tau_{Q}\underline{s_{r+2}})\nonumber\\
 &=& \widetilde{B}\underline{v}-(\widetilde{I}-\widetilde{R})\underline{\delta}-(\widetilde{R}^{tr}-\widetilde{R})\underline{s_{r+2}}+
  (\widetilde{I}-\widetilde{R})(\underline{s_{r+2}}+\tau_{Q}\underline{s_{r+2}})
\nonumber\\
&=&
\widetilde{B}\underline{v}-(\widetilde{I}-\widetilde{R})\underline{\delta}+(\widetilde{I}-\widetilde{R}^{tr})\underline{s_{r+2}}+
  (\widetilde{I}-\widetilde{R})\tau_{Q}\underline{s_{r+2}}
\nonumber\\
 &=&\widetilde{B}\underline{v}-(\widetilde{I}-\widetilde{R})\underline{\delta}.\nonumber
\end{eqnarray}
\end{proof}

\subsection{The case in  type  $\widetilde{D}_{r}$}

Let $Q$ be a tame quiver of type $\widetilde{D}_r$ for $r\geq 4$ as
follows
$$
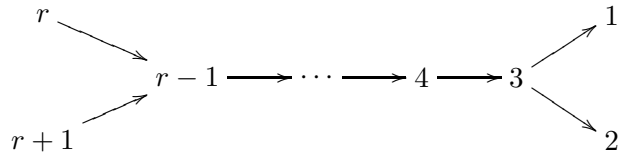
\xymatrix@R=0.8pc{r\ar[rd]&&&&&1\\
&r-1\ar[r]&\cdots\ar[r]&4\ar[r]&3\ar[ru]\ar[rd]&\\
r+1\ar[ru]&&&&&2}
$$
There are three non-homogeneous tubes (denoted by $\mathcal{T}_0,
\mathcal{T}_1, \mathcal{T}_{\infty}$). The minimal imaginary root of
$Q$ is $\delta=(1,1,2,\cdots,2,1,1).$ Let $E(\lambda)$ be the
regular simple module in the homogeneous tube for $\lambda\in
\mathbb{P}^1$ of degree $1$.
 The regular simple modules in $\mathcal{T}_1$
are
$$
\xymatrix@R=0.6pc{&k\ar[rd]^{1}&&&&k\\
E^{(1)}_t=S_{t+2} \mbox{ for } 1\leq t \leq r-3,
\quad E^{(1)}_{r-2}=:&&k\ar[r]^1&\cdots\ar[r]^1&k\ar[ru]^{1}\ar[rd]_{1}&\\
&k\ar[ru]_{1}&&&&k}
$$
By \cite{DXX}, we have
$$
|Gr_{\underline{e}}(E^{(1)}_1[r-2])|=|Gr_{\underline{e}}(E(\lambda))|+|Gr_{\underline{e}-\mathrm{\underline{dim}}S_3}(E^{(1)}_2[r-4])|.
$$

Similar to Lemma \ref{k1}, we have
\begin{proposition}\label{k3}
$$X_{E^{(1)}_1[r-2]}=X_{E(\lambda)}+q^{\frac{1}{2}}X_{E^{(1)}_2[r-4]}.$$
\end{proposition}

As a direct corollary of Proposition \ref{k1} and \ref{k3}, we
obtain the following refinement of Theorem \ref{affine}.
\begin{corollary}
Let  $\widetilde{Q}$ be reachable from $\widetilde{Q'}$ for any
graded quiver $Q'$ of type $\widetilde{A}$ and $\widetilde{D}$. Then
the set
$$\mathcal{B}(Q):=\{X_{\underline{d}}|\underline{d}\in \mathbb{Z}^{Q_0}\}$$
is a $\mathbb{ZP}$-basis of $\Acal_{|k|}(\widetilde{Q})$.
\end{corollary}
Now  assume that $Q$ is a tame quiver. Let $E$ be a regular simple
module in a nonhomogeneous tube with rank $n$ and $E(\lambda)$  the
regular simple module in the homogeneous tube for $\lambda\in
\mathbb{P}^1$ of degree $1$. We conjecture that:
\begin{conjecture}
With the above notations, we have
$$X_{E[n]}=X_{E(\lambda)}+q^{\frac{1}{2}}X_{(\tau_{Q}^{-1} E)[n-2]}.$$
\end{conjecture}

\section*{Acknowledgements}
 The authors would like
 to  thank Professor Bangming Deng and Dr. Fan Qin for many helpful discussions and suggestions.


\end{document}